\documentclass[twocolumn]{autart}  
\usepackage{graphicx}          
\usepackage{amsmath}
\usepackage{amssymb}
\usepackage{mathtools}
\usepackage{bm}
\usepackage{bbm}
\usepackage{mathrsfs}
\usepackage[sort&compress,square,numbers]{natbib}
\usepackage{xcolor}
\usepackage[pdftex, bookmarks=true, colorlinks=true, linkcolor=black, urlcolor=black, citecolor=black, linktoc=section]{hyperref}

\usepackage[normalem]{ulem}

\usepackage{enumerate}

\usepackage{subcaption}

\usepackage{amsfonts}
\usepackage{mathrsfs}

\usepackage{url}

\usepackage{enumerate}

\usepackage{graphicx}
\graphicspath{{./}{./photos}{./figures}}

\usepackage[noadjust]{cite}

\newtheorem{theorem}{Theorem}
\numberwithin{theorem}{section}
\newtheorem{lemma}[theorem]{Lemma}
\newtheorem{proposition}[theorem]{Proposition}
\newtheorem{remark}[theorem]{Remark}

\newtheorem{assumption}[theorem]{Assumption}
\newtheorem{corollary}[theorem]{Corollary}

\newtheorem{problem}[theorem]{Problem}

\let\oldpf\pf
\let\oldendpf\endpf
\def\pf{\begingroup \oldpf}
\def\endpf{\hfill~\qed \oldendpf \endgroup}

\expandafter\let\csname endpf*\endcsname=\endpf

\newenvironment{proof}[1][Proof.]{\begin{pf*}{\MakeUppercase{#1}}}{\end{pf*}}

\newcommand{\mR}{\mathbb{R}}
\newcommand{\mS}{\mathbb{S}}

\newcommand{\st}{\text{s.t.}}

\DeclareMathOperator*{\im}{Im}

\DeclareMathOperator*{\rank}{rank}

\DeclareMathOperator*{\tr}{tr}
\DeclareMathOperator*{\mE}{\mathbb{E}}
\DeclareMathOperator*{\mP}{\mathbb{P}}

\DeclareMathOperator*{\cov}{cov}

\newcommand{\Acl}{A_{cl}}

\newcommand{\zero}{\mathbf{0}}

\newcommand{\rvv}{\bm{v}}

\newcommand{\Qtrue}{\bar{Q}}
\newcommand{\qtrue}{\bar{q}}

\newcommand{\Rtrue}{\bar{R}}
\newcommand{\Ptrue}{\bar{P}}
\newcommand{\etatrue}{\bar{\eta}}
\newcommand{\xitrue}{\bar{\xi}}

\newcommand{\mfS}{\mathfrak{S}}
\newcommand{\mfR}{\mathfrak{R}}
\newcommand{\mfStrue}{\bar{\mfS}}
\newcommand{\mfRtrue}{\bar{\mfR}}

\newcommand{\mscrF}{\mathscr{F}}

\newcommand{\mfx}{\bm{x}}
\newcommand{\mfy}{\bm{y}}
\newcommand{\mfz}{\bm{z}}
\newcommand{\mfw}{\bm{w}}
\newcommand{\mfu}{\bm{u}}
\newcommand{\mfN}{\bm{N}}

\newcommand{\mfY}{\bm{Y}}

\newcommand{\g}{g}
\newcommand{\gtrue}{\bar{\g}}

\newcommand{\betatrue}{\bar{\beta}}
\newcommand{\gammatrue}{\bar{\gamma}}
\newcommand{\Htrue}{\bar{H}}

\edef\endfrontmatter{%
  \unexpanded\expandafter{\endfrontmatter}
  \noexpand\endNoHyper 
}

\begin{document}

\begin{frontmatter}
  \title{Statistically consistent inverse optimal control for discrete-time indefinite linear-quadratic systems\thanksref{footnoteinfo}}

  \thanks[footnoteinfo]{This work was partially supported by National Natural Science Foundation (NNSF) of China under Grant 62103276, and partially by the Wallenberg AI, Autonomous Systems and Software Program (WASP) funded by the Knut and Alice Wallenberg Foundation.}
  
  \author[SJTU]{Han Zhang}\ead{zhanghan\_tc@sjtu.edu.cn},
  \author[CHALMERS_AND_GU]{Axel Ringh}\ead{axelri@chalmers.se}
  
  \address[SJTU]{Department of Automation, School of Electronic Information and Electrical Engineering, Shanghai Jiao Tong University, Shanghai, China}
  \address[CHALMERS_AND_GU]{Department of Mathematical Sciences, Chalmers University of Technology and University of Gothenburg, 41296 Gothenburg, Sweden}  


  \begin{abstract}
The Inverse Optimal Control (IOC) problem is a structured system identification problem that aims to identify the underlying objective function based on observed optimal trajectories. This provides a data-driven way to model experts' behavior. In this paper, we consider the case of discrete-time finite-horizon linear-quadratic problems where: the quadratic cost term in the objective is not necessarily positive semi-definite; the planning horizon is a random variable; we have both process noise and observation noise; the dynamics can have a drift term; and where we can have a linear cost term in the objective. In this setting, we first formulate the necessary and sufficient conditions for when the forward optimal control problem is solvable. Next, we show that the corresponding IOC problem is identifiable. 
Using the conditions for existence of an optimum of the forward problem, we then formulate an estimator for the parameters in the objective function of the forward problem as the globally optimal solution to a convex optimization problem, and prove that the estimator is statistical consistent. Finally, the performance of the algorithm is demonstrated on two numerical examples.
  \end{abstract}

\end{frontmatter}

\section{Introduction}\label{sec:introduction}
Optimal control is a powerful framework in which control decisions are performed in order to minimize some given objective function; see, e.g., one of the monographs \citep{anderson2007optimal, bertsekas2000dynamic}. In fact, many processes in nature can be modelled as optimal control problems with respect to some criteria \citep{alexander1996optima}. However, in applications of optimal control, a fundamental problem is to design an appropriate objective function. In order to induce an appropriate control response, the object function needs to be adapted to the contextual environment in which the system is operating. This is a difficult task, which relies heavily on the designers' experience and imagination.

Instead of designing the cost criteria,
one way to overcome this difficulty would be to  identify the cost function from the observations of an expert system that behaves ``optimally" in the environment and thus ``imitating" the expert behaviour.
The latter is known as Inverse Optimal Control (IOC) \citep{kalman1964linear}, and has received considerable attention.  
In particular, IOC reconstructs the objective function of the expert system and hence predicts the closed-loop system's behaviour using observed data as well as the knowledge of underlying system dynamics. The problem can be categorized as gray-box system identification \citep[p.~13]{ljung1999system}, and is also closely related to inverse reinforcement learning \citep{ng2000algorithms}.

As one of the most classical optimal controller designs, linear-quadratic optimal regulators has been widely used in engineering. Though most of the literature considers the case of $Q$ (the penalty parameter for the states) being positive semi-definite, indefinite linear-quadratic optimal control \citep{chen1998stochastic,ran1993linear,rami2002discrete, ferrante2015note, ferrante2016discussion} has found applications in, e.g., mathematical finance \citep{zhou2000continuous},  crowd evacuation \citep{toumi2020tractable, toumi2021spatial}, and controller design for automatic steering of ships \citep{reid1983optimal}. To this end, we are motivated to develop an IOC framework for general indefinite linear-quadratic optimal control.
Linear-quadratic IOC problem has been studied under many different settings, including the infinite-horizon case in both continuous time \citep{boyd1994linear, anderson2007optimal} and discrete time \citep{priess2014solutions}, respectively, as well as the finite-horizon case in both continuous time \citep{li2018convex, li2020continuous} and discrete time \citep{keshavarz2011imputing, zhang2019inverse, yu2021system, zhang2021inverse, zhang2022statistically}, respectively. 
IOC is also closely connected to inverse reinforcement learning \citep{ng2000algorithms}, and this perspective has been used in \citep{xue2021inverse, lian2021robust, xue2021inverse_tracking} to consider infinite horizon discrete-time and continuous-time linear-quadratic set-ups for regulation, tracking, and adversary scenarios, respectively. 
However, to the best of our knowledge, IOC frameworks for general indefinite linear-quadratic optimal control has not been considered yet.
Furthermore, there are also other important aspects that have not been fully investigated in the aforementioned literature.
More precisely, any real-world data would inevitably contain noise: it can either be process noise, observation noise, or both. Therefore, from robustness and accuracy perspectives, it is important to have a statistically consistent estimator, i.e., that converges to the true underlying parameter values as the number of observation grows.
Moreover, many of the aforementioned IOC algorithms that are based on optimization either suffer from the fact that the estimation problems are nonconvex  \citep{keshavarz2011imputing, zhang2019inverse, yu2021system}, and can therefore have issues with local minima, or suffer from the fact the estimation procedure needs to know the control gain a priori \citep{boyd1994linear, anderson2007optimal, priess2014solutions, li2018convex, li2020continuous}. In the latter case, the estimation normally needs to be done in a two-stage procedure and the information is thus not used in the most efficient way.
Furthermore, most of the literature on linear-quadratic IOC consider the regulation problem. However, in many experimental set-ups, an expert agent may have more complicated tasks than regulation, e.g., tracking a reference signal.
Finally, real-world data can be of different time lengths, and this needs to be handled in a systematic way in order not to deteriorate the estimates.

In this work, we address these issues. More specifically, we extend our previous work \citep{zhang2021inverse} and the conference version \citep{zhang2022statistically} of this work, and consider the generalized linear-quadratic,  indefinite, discrete-time IOC problem.  
In particular, in \citep{zhang2021inverse}, we only consider the IOC problem of linear quadratic regulator with no process noise, and there is no linear state terms in the objective function.  Moreover, in that paper the parameter $Q$ is positive semi-definite, and the time-horizon length is fixed.  On the other hand, compared to the conference version \citep{zhang2022statistically},  we further consider the case that involves observation noise and indefinite matrix $Q$ in the objective function.

The contribution of this work is three-fold:
\begin{enumerate}
\item We give necessary and sufficient conditions for the well-posedness of the generalized discrete-time finite-horizon indefinite linear-quadratic optimal control problem. Specifically, we do not assume that the running cost is positive semi-definite, and we include a linear term in the objective function, a forcing term in the dynamics, and process noise.
\item We prove the identifiability of the parameters in the objective function.
\item We construct an IOC algorithm that works for both the positive semi-definite and the indefinite case. The algorithm is based on convex optimization, and we show that its unique optimal solution is the ``true" underlying parameters in the objective function of the forward problem.  Moreover, the constructed optimization problem, which contains an expectation in the objective function, is approximated by the empirical mean in practice and we also show that the estimator based on this approximation is statistically consistent. This means that the estimator converge in probability to the true underlying parameter when the number of observations goes to infinity. In addition, the convex optimization formulation guarantees that the statistically consistent estimate that corresponds to the global optimum can actually be attained in practice.
\end{enumerate}

The article is organized as follows. In Section \ref{sec:problem_formulation}, we formulate the forward and inverse problem, and specify the assumptions we use. Next, in Section \ref{sec:forward_problem}, we analyze the forward problem and prove the necessary and sufficient conditions for when it is well-posed.  Section \ref{sec:identifiability} investigates the identifiability issue of the formulated IOC problem, and in particular prove that the formulated IOC problem is identifiable. 
In Section \ref{sec:IOC_algorithm}, we construct the IOC algorithm, construct the estimator, and prove that the latter is statistical consistent. To illustrate the performance of the proposed algorithm, in Section \ref{sec:numerical_examples} we include some numerical examples. Finally, we conclude the paper in Section \ref{sec:conclusion}. For improved readability, some proofs are deferred to the Appendix.

\textit{Notation:} $\mathbb{S}^n$ denotes the set of $n\times n$ symmetric matrices, while $\mathbb{S}^n_+$ denotes the set of $n\times n$ positive semi-definite matrices. $\|\cdot\|$ denotes $l_2$-norm and $\|\cdot\|_F$ denotes Frobenius norm.
$\mathscr{B}^{n}_{\varphi}(x) := \{ y \in \mR^n \mid \| x - y \| < \varphi \}$ denotes the open ball of radius $\varphi$ centered at $x$, and $\bar{\mathscr{B}}^{n}_{\varphi}(x)$ denotes its closure. Moreover, $\mathbb{S}^n_+(\varphi) := \{ G \in \mathbb{S}^n_+ \mid \| G \|_F \leq \varphi \}$.
For $G_1,G_2\in\mathbb{S}^n$, we denote $G_1\succeq G_2$ as Loewner partial order of $G_1$ and $G_2$, i.e., $G_1-G_2\in\mathbb{S}^n_+$. Further, $G^\dagger$ denotes the Moore-Penrose pseudo-inverse of $G$.
For a square matrix $G = \left[ \begin{smallmatrix} G_{11} & G_{12} \\ G_{21} & G_{22} \end{smallmatrix} \right]$, where $G_{11}$ and $G_{22}$ are square, we define the Schur complements $G \backslash G_{22} = G_{11} - G_{21} G_{22}^\dagger G_{12}$ and $G \backslash G_{11} = G_{22} - G_{12} G_{11}^\dagger G_{21}$ (see, e.g., \citep[Sec.~1.6]{horn2005basic}).
For integers $\nu_1, \nu_2$, we let $\nu_1:\nu_2$ denote $\nu_1, \nu_1+1, \ldots, \nu_2$, and define the expression as empty if $\nu_2 < \nu_1$.
In addition, by $\im$ and $\ker$ we denote the image space and kernel, respectively, by $^\perp$ we denote the orthogonal complement, and by $\neg$ we denote the (mathematical) negation of a statement. Finally, we use \textbf{\textit{italic bold font}} to denote stochastic elements, and use $\overset{p}\rightarrow$ to denote convergence in probability, i.e., for a sequence of random elements $\{ \bm{a}^i \}_ {i = 1}^\infty$ and a random element $\bm{a}$, $\bm{a}^i \overset{p}\rightarrow \bm{a}$ means that for all $\varepsilon > 0$, $\lim_{i \to \infty} \mP(\|\bm{a}^i - \bm{a}\| > \varepsilon) = 0$.

\section{Problem formulation}\label{sec:problem_formulation}
In this section, we introduce the forward problem as well as the inverse problem.
To solve the inverse problem, we use measured (noisy) optimal trajectories from an expert that performs the given task multiple times. This gives multiple demonstration trajectories that can be used to learn the cost.
However, while the underlying ``decision principle" is the same in all the demonstration trajectories, these trajectories may have different lengths.
We are hence motivated to cover those demonstrations with one mathematical formulation.

We start by introducing the mathematical formulation of the forward optimal control problem. To this end, let $(\Omega,\mathcal{F},\mathbb{P})$ be a probability space that carries a random vector $\bar{\mfx}\in\mR^n$, stochastic processes $\{\mfw_t\in\mR^n\}_{t=1}^\infty$, $\{\rvv_t\in\mR^n\}_{t=1}^\infty$ (the measurement noise to appear in \eqref{eq:noisy_observation}), and a random variable $\mfN\in\{2, 3, \cdots,\nu\}\subset\mathbb{Z}_+$.
It is assumed that for each realization $(\bar{x},N)$ of the random element $(\bar{\mfx},\mfN)$ (corresponding to the initial position and planning horizon length), the agent's control decision $\mfu_t$ is determined by a stochastic generalized linear-quadratic control problem, namely,
\begin{subequations}\label{eq:stochastic_forward_problem}
\begin{align}
\min_{\substack{\mfx_{1:\nu} , \\ \mfu_{1:\nu}}}
  \;
  & \; J_N := \mE_{\mfw_{\nu-N+1:\nu-1}} \Big[\frac{1}{2} \mfx_{\nu}^{T} \Qtrue \mfx_{\nu} +\qtrue^T\mfx_{\nu}  \nonumber\\
  & \;\; +\sum_{t = \nu - N+1}^{\nu-1} [\frac{1}{2} \mfx_{t}^{T} \Qtrue \mfx_{t} +\qtrue^T\mfx_t + \frac{1}{2}\mfu_{t}^T\Rtrue\mfu_t ] \Big]\label{eq:stochastic_forward_problem_cost}\\
  \st
  & \; \mfx_{t+1} = A\mfx_{t} + B\mfu_{t}+d+\mfw_t, \nonumber\\
  &\quad t = \nu\!-\!N\!+\!1:\nu\!-\!1, \label{eq:stochastic_forward_problem_dynamics} \\
  & \; \mfx_{t + 1} = \mfx_t, \quad t=1:\nu-N\label{eq:stochastic_forward_problem_before_init_cond} \\
  & \; \mfx_1 = \bar{x}, \label{eq:stochastic_forward_problem_init_cond}\\
  & \; \mfu_1=\ldots=\mfu_{\nu-N} = 0,\label{eq:stochastic_control_stand_still}
\end{align}
\end{subequations}
where $A \in \mR^{n \times n}$, $B \in \mR^{n \times m}$, $\Qtrue\in\mathbb{S}^n$, $\Rtrue\in\mathbb{S}^m$, and $\qtrue,d \in\mathbb{R}^n$.
More specifically, the minimization in \eqref{eq:stochastic_forward_problem} is over admissible control strategies with complete state information, i.e., $\mfu_t$ is a function that maps from $\mR^n$ to $\mR^m$, $\mfu_t : \mfx_{t} \mapsto \mfu_t(\mfx_{t})$ (see, e.g., \citep[Chp.~8]{astrom1970introduction}).

The formulation \eqref{eq:stochastic_forward_problem} is motivated by the fact that an agent can have different time horizon lengths (and initial values) to complete different tasks,
while the underlying decision principle (i.e., running cost) remains unchanged since the principle is connected to the agent's characteristics.
In particular, given a realization of the time-horizon length $\mfN=N$ and initial value $\bar{\mfx} = \bar{x}$, the agent starts to apply its control from the initial value $\bar{x}$ at the time instant $t=\nu-N+1$ and the agent maintains \emph{the same running cost} for each control execution.
This formulation gives a systematic way to handle real-world data with different time lengths (see \citep{zhang2022statistically}).
Moreover, note that since the dynamics \eqref{eq:stochastic_forward_problem_dynamics} and the running cost in \eqref{eq:stochastic_forward_problem_cost} are time-invariant, by Bellman's principle of optimality, \eqref{eq:stochastic_forward_problem} can be reformulated to optimal control problems with planning horizon length $N$ and that start to control from an initial state at time point $t=1$. However, it turns out to be convenient to align the optimal demonstration trajectories with different lengths at the end time point, and view the demonstration trajectories as if the expert starts to control the system at different time instants, i.e., to formulate the problem as in \eqref{eq:stochastic_forward_problem}.

\begin{remark}\label{rem:forward_problem_formulation_rem}
The reason why we only consider a time-invariant external forcing term $d$ in \eqref{eq:stochastic_forward_problem_dynamics} is to simplify the presentation. The results of the paper also hold for time-varying forcing terms $d_t$, provided that the agent knows all the future forcing terms for $t = \nu-N+1:\nu-1$.
Similarly, this formulation can be extended to tracking-problems by letting the linear cost term be time-varying, $q_t = -Qx_t^r$ where $x_t^r$ is the reference signal, and the results in the paper follows analogously (cf.~\citep{zhang2022statistically}). Notably, with time-varying $d_t$ or $q_t$, the arguments in the paragraph just before (using time-invariance of running cost and dynamics) does not hold. Nevertheless, the formulation \eqref{eq:stochastic_forward_problem} is still of interest in application such as the tracking in rehabilitation trainings, see \citep{zhang2022statistically}.
\end{remark}

In the remainder of the paper, we make the following assumptions.

\begin{assumption}[Controlability and full rank]\label{ass:controlability_and_full_rank}
The system $(A,B)$ is controllable, $A$ is invertible, and $B$ is of full-column rank.
\end{assumption}

\begin{assumption}[Independent white random elements]\label{ass:IID}
The discrete time stochastic processes $\{\mfw_t\}_{t=1}^\infty$ and $\{\rvv_t\}_{t=1}^\infty$ are independent zero-mean white-noise processes. More specifically, this means that $\mE[\mfw_t]=0$, $\mE[\rvv_t]=0$, $\forall t$, and $\cov(\mfw_t,\mfw_s) = \Sigma_w \delta(t - s)$, $\cov(\rvv_t,\rvv_t)=\Sigma_v \delta(t - s)$, where $\delta(t)$ is the Dirac-delta function, and where $\Sigma_w \succeq 0$ and $\Sigma_v \succeq 0$ are a priori known.
Moreover, the random element $(\bar{\mfx},\mfN)$ is independent of the two stochastic processes.
\end{assumption}

The rationale behind the first assumption is that we are considering a controllable discrete-time linear system that is not over-actuated,%
\footnote{It means that, given the system state's evolution from time step $t$ to $t+1$, there is only one possible control input to realize that.}
and that is sampled from a continuous-time system.
Next, make the following mild assumptions regarding the stochastic planning horizon $\mfN$.

\begin{assumption}[Support of the planning horizon]\label{ass:planning_horizon}
The constant $\nu \in \mathbb{Z}_+$ is known, and $\nu \ge n+1$. Moreover, the probability distribution for $\mfN$ satisfies $\mP(\mfN \in [2, \nu])=1$, and $\mP(\mfN = \nu) > 0$. 
\end{assumption}

The above assumption means that the longest possible planning horizon is known, that it is sufficiently long, and that the longest horizon $\nu$ can be realized, i.e., it has a nonzero probability.
Next, note that since $\Qtrue$ might not be positive semi-definite, \eqref{eq:stochastic_forward_problem} might not admit an optimal solution (see, e.g. \citep{ferrante2015note,ferrante2016discussion}).
We analyze the well-posedness of the forward problem  \eqref{eq:stochastic_forward_problem} in depth in Section~\ref{sec:forward_problem}. However, before that, we have the following proposition which illustrates the reason why we emphasize the longest time-horizon length in Assumption \ref{ass:planning_horizon}.

\begin{proposition}\label{prop:stability_longest_horizon}
Under Assumptions~\ref{ass:controlability_and_full_rank} and \ref{ass:IID}, if the optimal control problem \eqref{eq:stochastic_forward_problem} with the objective function given by $(\Qtrue, \qtrue, \Rtrue)$ admits a solution for planning horizon $N = \nu$ for any $\bar{x}\in\mR^n$, then it admits a solution for all $N = 2:\nu$ for any $\bar{x}\in\mR^n$.
\end{proposition}

\begin{proof}
See appendix.
\end{proof}

With Assumption~\ref{ass:planning_horizon} and Proposition~\ref{prop:stability_longest_horizon} in mind, we are thus interested in parameters that belong to the following set:
\begin{align*}
&\mscrF(\Rtrue) =  \{ (Q, q) \in \mS^n \times \mR^n  \mid  \text{the optimal control problem} \\
& \text{\eqref{eq:stochastic_forward_problem} with the objective function given by $(Q, q, \Rtrue)$ admits}\\
& \text{solutions a.s.~under the distribution of $\bar{\mfx}$ and for all}\\
& N \in \{2, 3, \ldots,\nu\}\}. 
\end{align*}
More specifically, if $(\Qtrue, \qtrue) \in \mscrF(\Rtrue)$, then the forward problem is well-posed for any $N\in\{2,\ldots,\nu\}$ under the distribution of $\bar{\mfx}$ almost surely.
In addition, for the inverse problem we assume that the observation of the optimal states $\mfx_{1:\nu}$ are contaminated by observation noise:
\begin{align}\label{eq:noisy_observation}
\mfy_t = \mfx_t+\rvv_t,\quad t=1:\nu,
\end{align}
and that we observe $M$ trials of the agent. More precisely, let $\mfy_t^i$ have the same distribution as $\mfy_t$, for all $i = 1:M$ and all $t = 1:\nu$.
Then the observed $M$ trajectories of the agent's trials, $\{y_t^i\}_{i=1}^M$, are just realizations of the  I.I.D.~random vectors $\{\mfy_t^i\}_{i=1}^M$.

Before we formulate the IOC problem, we also make the following assumption.

\begin{assumption}[Initial value distribution]\label{ass:persistent_excitation}
The random element $(\bar{\mfx},\mfN)$ is such that $\mE[\|\bar{\mfx}\|^2]<\infty$. Moreover, for all $N \in \{ 2, \ldots, \nu\}$ such that $\mP(\mfN = N) > 0$,
it holds that for all $\chi \in \mR^n$, there exists a $\rho >  0$ such that $\mP(\bar{\mfx} \in \mathscr{B}^{n}_{\epsilon}(\rho \chi) \mid \mfN = N) >0$ for all $\epsilon > 0$.
\end{assumption}

Intuitively speaking, the above assumption states that, for each planning horizon length of interest, the initial value for the forward problem can be in any ``direction'' from the origin. 
This turns out to be important for both the forward and the inverse problem. The latter will be discussed in Section~\ref{sec:identifiability}.

With the setup presented in this section, we summarize the IOC problem to be considered in this paper. For the sake of simplicity,  we consider the case $\Rtrue=I$ when designing the IOC algorithm.

\begin{problem}[General stochastic linear-\\ quadratic IOC]
Suppose the unknown $(\Qtrue,\qtrue) \in \mscrF(I)$.
Given the optimal state trajectory observations $\{y_t^i\}_{t=1}^{\nu}$ of the agent's trials $i=1:M$ that are governed by \eqref{eq:stochastic_forward_problem}, estimate the corresponding $(\Qtrue,\qtrue)$ in the objective function \eqref{eq:stochastic_forward_problem_cost} that governs the agents' motion.
\end{problem}

\section{Forward problem analysis}\label{sec:forward_problem}

Before we present the IOC algorithm set-up, we first need to analyze the forward problem. More precisely, we need to characterize the set $\mscrF(\Rtrue)$ and find the necessary and sufficient optimality conditions for the existence of such generalized indefinite linear-quadratic optimal control. 
This is not only because we want to ensure that the forward-problem is well-behaved, but also to construct the IOC algorithm based on the optimality conditions.
Moreover, we also analyze the properties of the time-varying closed-loop system matrices that will be useful in developing the IOC algorithm.
For the theoretical development in this section, we do not assume that $\Rtrue=I$.

\subsection{Necessary and sufficient conditions for existence of optimal control}\label{sec:stochastic_optimality_conditions}

To this end, we first derive the necessary and sufficient conditions for existence of optimal control to \eqref{eq:stochastic_forward_problem}.
The results build on \citep{ferrante2015note}; in particular, some of the proof ideas are inspired by the proof in \citep[Thm.~2.1]{ferrante2015note}.
However, we not only extend the result to a more general setting of stochastic linear-quadratic problems, but also show that solvability of the forward problem, i.e., that $(\Qtrue,\qtrue) \in \mscrF(\Rtrue)$, can be characterized in different, but equivalent, ways. The main result of this section is the following.
\begin{theorem}[Boundedness of forward problem]\label{thm:indefinite_LQR}
Under Assumptions~\ref{ass:IID} and \ref{ass:persistent_excitation}, the following statements are equivalent:
\begin{enumerate}
\item $(\Qtrue,\qtrue) \in \mscrF(\Rtrue)$.
\item Let $\Ptrue_{1:\nu}$ and $\etatrue_{1:\nu}$ be generated by the following Riccati iterations
\begin{subequations}\label{eq:generalized_riccati_iterations}
\begin{align}
 \Ptrue_{\nu} & = \Qtrue, \label{eq:generalized_riccati_1}\\
\Ptrue_t & =  A^T \Ptrue_{t+1} A  + \Qtrue - A^T \Ptrue_{t+1} B (B^T \Ptrue_{t+1} B  + \Rtrue)^\dagger\nonumber \\
 & \; \times B^T \Ptrue_{t+1} A, \quad t=1:\nu-1; \label{eq:generalized_riccati_2}\\
 \etatrue_{\nu} &= \qtrue,\label{eq:generalized_riccati_3}\\
 \etatrue_t &= \left(A-B(B^T\Ptrue_{t+1}B+\Rtrue)^\dagger B^T\Ptrue_{t+1}A\right)^T\nonumber\\
 &\:\times(\etatrue_{t+1}+\Ptrue_{t+1}d)+\qtrue,\quad t=1:\nu-1.\label{eq:generalized_riccati_4}
\end{align}
\end{subequations}
Denote
\begin{subequations}\label{eq:psd_kernel_containment_cond}
\begin{align}
\mfStrue_t &:= B^T\Ptrue_{t+1}A, \label{eq:mfS_true}\\
\mfRtrue_t &:=B^T\Ptrue_{t+1}B+\Rtrue, \label{eq:mfR_def} \\
\gtrue_t &:= B^T\etatrue_{t+1}+B^T \Ptrue_{t+1}d \label{eq:mfg_true}.
\end{align}
It holds that
\begin{align}
\mfRtrue_t&\succeq 0,\label{eq:psd_iff_cond}\\
\ker(\mfRtrue_t)&\subset \Big[\ker(\mfStrue_t^T)\cap \ker(\gtrue_t^T)\Big], \label{eq:kernel_containment}
\end{align}
\end{subequations}
for all $t=1:\nu-1$.
\item There exists $\{ \Ptrue_{t} \in \mathbb{S}^n \}_{t = 1:\nu}$, $\{ \etatrue_{t} \in\mathbb{R}^{n} \}_{t = 1:\nu}$, and $\{ \xitrue_t \in \mR \}_{t= 1 : \nu}$ such that
\begin{subequations}\label{eq:nes_suff_existence_solution}
\begin{align}
& \Ptrue_{\nu} = \Qtrue, \quad \etatrue_{\nu}=\qtrue,\label{eq:nes_suff_existence_solution_1} \\
& \Htrue_t := 
\begin{bmatrix}
\mfRtrue_t & \mfStrue_t & \gtrue_t\\
\mfStrue_t^T & A^T\Ptrue_{t+1}A+\Qtrue-\Ptrue_t & \betatrue_t\\
\gtrue_t^T & \betatrue_t^T & \xitrue_t
\end{bmatrix} \succeq 0 \label{eq:nes_suff_existence_solution_2} \\
& \rank(\Htrue_t) = \rank(\mfRtrue_t) \label{eq:nes_suff_existence_solution_3}
\end{align}
\end{subequations}
where $\betatrue_t := \qtrue+A^T\Ptrue_{t+1}d+A^T\etatrue_{t+1}-\etatrue_t$ and where $\mfStrue_t$, $\mfRtrue_t$, and $\gtrue_t$ are as in \eqref{eq:mfS_true}, \eqref{eq:mfR_def}, and \eqref{eq:mfg_true}, respectively, for all $t = 1:\nu-1$.
\item The Hamilton-Jacobi-Bellman equation (HJBE)
\begin{subequations}\label{eq:stochastic_dyn_prog}
\begin{align}
&V_{\nu}(\chi_\nu) := \frac{1}{2}\chi_{\nu}^T\Qtrue \chi_{\nu}+\qtrue^T\chi_{\nu},\label{eq:stochastic_dyn_prog_terminal}\\
&V_t(\chi_t) =\min_{\mu_t} \Big\{\frac{1}{2}\chi_t^T \Qtrue \chi_t+\qtrue^T\chi_t + \frac{1}{2}\mu_t^T\Rtrue\mu_t \label{eq:stochastic_dyn_prog_iteration}\\
&+  \mE_{\mfw_t}\left[V_{t+1}(A\chi_t + B\mu_t + d+\mfw_t)\right] \Big\}, \:t=1:\nu-1\nonumber
\end{align}
\end{subequations}
has a solution. More precisely, this means that $V_t(\chi_t)$ is bounded from below for any $\chi_t\in\mR^n$, for $t=1:\nu-1$.
Moreover, the solution has the form 
\begin{equation}\label{eq:Vt}
V_{t}(\chi_t) =\frac{1}{2} \chi_t^T \Ptrue_t \chi_t+\etatrue_t^T\chi_t+\gammatrue_t, \quad t = 1:\nu,
\end{equation}
where $\Ptrue_{1:\nu}$ and $\etatrue_{1:\nu}$ are generated by \eqref{eq:generalized_riccati_iterations} and 
\begin{subequations}\label{eq:stochastic_riccati_iteration_constant_term}
\begin{align}
\gammatrue_{\nu} & = 0\\
\gammatrue_t & = \gammatrue_{t+1}-\frac{1}{2} \gtrue_t^T \mfRtrue_t^\dagger \gtrue_t+\frac{1}{2}d^T\Ptrue_{t+1}d+\etatrue_{t+1}^Td\nonumber\\
&\;\;\;\;+\frac{1}{2}\tr(\Ptrue_{t+1}\Sigma_w),\quad t=1:\nu-1,
\end{align}
\end{subequations}
with $\mfRtrue_{1:\nu-1}$ and $\gtrue_{1:\nu-1}$ as in \eqref{eq:mfR_def} and \eqref{eq:mfg_true}, respectively.
\end{enumerate}
In addition, if any of the four above conditions hold, then the optimal control signal $\mfu_t$ for \eqref{eq:stochastic_forward_problem} is parametrized by any arbitrary vector $\lambda_t \in\mathbb{R}^n $, and is given by
\begin{subequations}\label{eq:optimal_ctrl_formula}
\begin{align}
\mfu_t & = -\mfRtrue_t^\dagger(\mfStrue_t \mfx_t +\gtrue_t)+\mathcal{P}^{\ker(\mfRtrue)}_t \lambda_t, \\
&\quad\:t=\nu-N+1:\nu-1, \nonumber \\
\mfu_t &= 0,\:t=1:\nu-N,
\end{align}
\end{subequations}
where $\mathcal{P}^{\ker(\mfRtrue)}_{t} = I-\mfRtrue_{t}^\dagger\mfRtrue_{t}$ is the projection operator onto the kernel space of $\mfRtrue_{t}$.%
\footnote{For the fact that $\mathcal{P}^{\ker(\mfRtrue)}_{t} = I-\mfRtrue_{t}^\dagger\mfRtrue_{t}$ is indeed the projection operator onto the kernel space of $\mfRtrue_{t}$, see, e.g., \citep[Prop.~2.3]{engl2000regularization}.}
\end{theorem}

\begin{proof}
First, assume that 4) holds. Note that for an agent with a planning horizon realization $N=\nu$, by the principle of optimality in dynamic programming  (see, e.g., \citep[p.~18]{bertsekas2000dynamic}), we can easily show that 4)$\implies$1) (cf.~ \citep[Prop.~1.3.1]{bertsekas2000dynamic}). For the case $N < \nu$, note that the HJBE is still valid for $t = \nu - N +1:\nu$, and thus the agents behavior is still optimal in $t = \nu - N +1 : \nu$. Since the systems behavior for $t = 1 : \nu - N$  is completely determined by  \eqref{eq:stochastic_forward_problem_before_init_cond}, and \eqref{eq:stochastic_control_stand_still}, by Bellman's principle of optimality, the solution to the HJBE still gives the optimal control. This implies 1), and hence shows that 4)$\implies$1).

Next, we prove that 1)$\implies$2) by proving the contraposition of the statement, i.e., by proving that $\neg$2) $\implies\neg$1). 
Before we proceed, first note the fact that
\begin{align}
&\sum_{t=\nu - N + 1}^{\nu-1}\Big\{\frac{1}{2}\mfx_{t+1}^T \Ptrue_{t+1}\mfx_{t+1} + \etatrue_{t+1}^T\mfx_{t+1}\nonumber\\
&\quad -\frac{1}{2}\mfx_t^T \Ptrue_{t}\mfx_t - \etatrue_{t}^T\mfx_t \Big\} + \frac{1}{2}\mfx_{\nu-N+1}^T \Ptrue_{\nu-N+1} \mfx_{\nu-N+1} \nonumber\\
&\quad+ \etatrue_{\nu-N+1}^T \mfx_{\nu-N+1} - \frac{1}{2}\mfx_{\nu}^T \Ptrue_{\nu} \mfx_{\nu} - \etatrue_{\nu}^T\mfx_{\nu}=0,\label{eq:add_subtract_terms}
\end{align}
for any $N \in \{2,\ldots, \nu\}$.
Taking the expectation with respect to $\mfw_{\nu-N+1:\nu-1}$ on both sides of \eqref{eq:add_subtract_terms}, adding the latter expression to  \eqref{eq:stochastic_forward_problem_cost}, and using \eqref{eq:stochastic_forward_problem_dynamics}, \eqref{eq:generalized_riccati_1}, \eqref{eq:generalized_riccati_3}, and Assumption~\ref{ass:IID},
we can write the objective function as
\begin{align}
&J_N=\mE_{\mfw_{\nu-N+1:\nu-1}} \Big[\frac{1}{2} \mfx_{\nu}^T \underbrace{(\Qtrue - \Ptrue_{\nu})}_{=0}\mfx_{\nu}+\underbrace{(\qtrue-\etatrue_{\nu})^T}_{=0}\mfx_{\nu} \nonumber \\
&+\sum_{t=\nu-N+1}^{\nu-1}\Big\{\frac{1}{2}(A\mfx_t+B\mfu_t+d + \mfw_t)^T\Ptrue_{t+1}\nonumber\\
&\times(A\mfx_t+B\mfu_t+d + \mfw_t)+\etatrue_{t+1}^T(A\mfx_t+B\mfu_t+d + \mfw_t) \nonumber \\
&-\frac{1}{2}\mfx_t^T\Ptrue_t\mfx_t-\etatrue_t^T\mfx_t+\frac{1}{2}\mfx_t^T\Qtrue \mfx_t+\qtrue^T\mfx_t+\frac{1}{2}\mfu_t^T\Rtrue \mfu_t\Big\} \nonumber \\
&+\frac{1}{2}\mfx_{\nu-N+1}^T\Ptrue_{\nu-N+1}\mfx_{\nu-N+1}+\etatrue_{\nu-N+1}^T\mfx_{\nu-N+1} \Big]\nonumber\\
&=\mE_{\mfw_{\nu-N+1:\nu-1}} \Big[\sum_{t=\nu-N+1}^{\nu-1}\Big\{\frac{1}{2}
\begin{bmatrix}
\mfu_t^T &\mfx_t^T&1
\end{bmatrix}
\Htrue_{t}
\begin{bmatrix}
\mfu_t\\\mfx_t\\1
\end{bmatrix} \Big\} \nonumber \\
&+\frac{1}{2}\mfx_{\nu-N+1}^T\Ptrue_{\nu-N+1}\mfx_{\nu-N+1}+\etatrue_{\nu-N+1}^T\mfx_{\nu-N+1}\Big]  \nonumber\\
&+ \!\!\! \underbrace{\sum_{t=\nu-N+1}^{\nu-1} \frac{1}{2}d^T\Ptrue_{t+1}d  + \frac{1}{2}\tr({\Ptrue_{t+1} \Sigma_w}) + \etatrue_{t+1}^Td-\frac{1}{2}\xitrue_{t}}_{=: \; \Upsilon_t}  \label{eq:J_H}
\end{align}
with $\Htrue_t$ in the form of \eqref{eq:nes_suff_existence_solution_2} and $\xitrue_t=\gtrue_t^T\mfRtrue_t^\dagger \gtrue_t$ in $\Htrue_t$.
Note that the term $\Upsilon_t$ in the above equation is constant with respect to the state and the control and hence can be discarded in the optimization problem. On the other hand, since $\{\Ptrue_t\}_{t=1}^{\nu}$ and $\{\etatrue_t\}_{t=1}^{\nu}$ are generated by \eqref{eq:generalized_riccati_iterations}, $\Htrue_t$ can also be written as
\begin{align}\label{eq:H_t_g_t_rewriting}
\Htrue_t = \begin{bmatrix}
\mfRtrue_t &\mfStrue_t &\gtrue_t\\
\mfStrue_t^T &\mfStrue_t^T\mfRtrue_t^\dagger \mfStrue_t &\mfStrue_t^T\mfRtrue_t^\dagger \gtrue_t\\
\gtrue_t^T & \gtrue_t^T\mfRtrue_t^\dagger\mfStrue_t&\gtrue_t^T\mfRtrue_t^\dagger \gtrue_t
\end{bmatrix}.
\end{align}

Now, we use the above trick to prove that $\neg$2)$\implies \neg$1). Suppose \eqref{eq:psd_iff_cond} and \eqref{eq:kernel_containment} cease to hold at the $N$th backward iteration \eqref{eq:generalized_riccati_iterations}, where $N\in\{2,\ldots,\nu\}$.
Namely, $\mfRtrue_{\nu-N+t}\succeq 0$, $\ker(\mfRtrue_{\nu-N+t})\subset \left[\ker(\mfStrue_{\nu-N+t}^T)\cap \ker(\gtrue_{\nu-N+t}^{T})\right]$ still holds for $t=2:N$ but not for $t=1$.
We proceed by showing that this implies that for this planning horizon length realization $N$, \eqref{eq:stochastic_forward_problem} is not bounded from below.
In particular, by \eqref{eq:H_t_g_t_rewriting} and $\ker(\mfRtrue_{\nu-N+t})\subset \left[\ker(\mfStrue_{\nu-N+t}^T)\cap \ker(\gtrue_{\nu-N+t}^{T})\right]$, $\forall t=2:N$, it follows that
\begin{align}
&\Htrue_{\nu-N+t} = \begin{bmatrix}
\mfRtrue_{\nu-N+t} \\\mfStrue_{\nu-N+t}^T \\ \gtrue_{\nu-N+t}^T
\end{bmatrix}
\mfRtrue_{\nu-N+t}^\dagger\nonumber \\
&\qquad \times
\begin{bmatrix}
\mfRtrue_{\nu-N+t} & \mfStrue_{\nu-N+t} &\gtrue_{\nu-N+t}
\end{bmatrix}\succeq 0 \label{eq:H_decomp}
\end{align}
for $t=2:N$.
Hence, in view of \eqref{eq:stochastic_forward_problem_init_cond}, \eqref{eq:stochastic_forward_problem_before_init_cond} and the fact that $\{\mfw_t\}_{t=1}^\infty$ is independent of other random elements from Assumption \ref{ass:IID}, for the given ``initial state" realization $\mfx_{\nu-N+1}=\bar{x}\in\mR^n$ from which the agent starts tracking, the objective function can be written as
\begin{align}
&J_N=\mE_{\mfw_{\nu-N+1:\nu-1}} \Big[\frac{1}{2}
\begin{bmatrix}
\mfu_{\nu-N+1}^T &\mfx_{\nu-N+1}^T &1
\end{bmatrix}\Htrue_{\nu-N+1}
\nonumber\\
&\times \begin{bmatrix}
\mfu_{\nu-N+1} \\\mfx_{\nu-N+1}\\1
\end{bmatrix}+\sum_{t=2}^{N-1}\Big\{\frac{1}{2}
\Big\|(\mfRtrue_{\nu-N+t}^\dagger)^{\frac{1}{2}}\Big(\mfRtrue_{\nu-N+t}\mfu_{\nu-N+t}\nonumber\\
&+\mfStrue_{\nu-N+t}\mfx_{\nu-N+t}+\gtrue_{\nu-N+t}\Big)\Big\|^2\Big\}\nonumber\\
&+ \! \frac{1}{2}\underbrace{\mfx_{\nu-N+1}^T}_{\bar{x}^T} \! \Ptrue_{\nu-N+1} \! \underbrace{\mfx_{\nu-N+1}}_{\bar{x}} + \etatrue_{\nu-N+1}^T\!\underbrace{\mfx_{\nu-N+1}}_{\bar{x}}\Big].
\label{eq:objective_function_time_split}
\end{align}
Notably, it holds that
\begin{align*}
&\frac{1}{2}\bar{x}^T\Ptrue_{\nu-N+1}\bar{x}+\etatrue_{\nu-N+1}^T\bar{x}\le \frac{1}{2}\sigma_{\max}(\Ptrue_{\nu-N+1})\|\bar{x}\|^2\\
&+\|\etatrue_{\nu-N+1}\|\cdot\|\bar{x}\|:=\tau(\bar{x}),
\end{align*}
where $\sigma_{\max}(\cdot)$ is the largest eigenvalue of a matrix.
Also note that, for any $\mfu_{\nu-N+1}$ and $\mfx_{\nu-N+1}=\bar{x}$,  by the dynamics \eqref{eq:stochastic_forward_problem_dynamics} we get an $\mfx_{\nu-N+2}$. Selecting $\mfu_{\nu-N+2} = -\mfRtrue_{\nu-N+2}^\dagger(\mfStrue_{\nu-N+2} \mfx_{\nu-N+2} + \gtrue_{\nu-N+2})+\mathcal{P}_{\nu-N+2}^{\ker(\mfRtrue)}\lambda_{\nu-N+2}$, for some arbitrary $\lambda_{\nu-N+2}\in\mathbb{R}^n$, this would give the next state $\mfx_{\nu-N+3}$.
Recursively selecting the other control signals for $t=\nu-N+3:\nu-1$ accordingly, i.e., as $\mfu_t = -\mfRtrue_t^\dagger(\mfStrue_t \mfx_t + \gtrue_t)+\mathcal{P}_t^{\ker(\mfRtrue)}\lambda_t$, for some arbitrary $\lambda_t\in\mathbb{R}^n$,  then all terms in the summation in \eqref{eq:objective_function_time_split} would become zero.
Therefore, for the objective function $J_N$ it holds that
\begin{align}
&J_N\le \mE_{\mfw_{\nu-N+1:\nu-1}} \Big[\frac{1}{2}\mfu_{\nu-N+1}^T\mfRtrue_{\nu-N+1}\mfu_{\nu-N+1} \nonumber\\
&+\bar{x}^T\mfStrue_{\nu-N+1}^T\mfu_{\nu-N+1}+\gtrue_{\nu-N+1}^T\mfu_{\nu-N+1}\Big] \nonumber\\
&+\bar{x}^T\mfStrue_{\nu-N+1}^T\mfRtrue_{\nu-N+1}^\dagger\mfStrue_{\nu-N+1}\bar{x} \nonumber\\
&+\gtrue_{\nu-N+1}^T\mfRtrue_{\nu-N+1}^\dagger \mfStrue_{\nu-N+1}\bar{x}+\tau(\bar{x}). \label{eq:JN_inequality}
\end{align}
If $\mfRtrue_{\nu-N+1}\not\succeq 0$, it is clear that we can choose $\mfu_{\nu-N+1}=\alpha v_-(\mfRtrue_{\nu-N+1})$, where $v_-(\mfRtrue_{\nu-N+1})$ is an eigenvector that corresponds to a negative eigenvalue of $\mfRtrue_{\nu-N+1}$. In this case, since such choice of $\mfu_{\nu-N+1}$ does not depend on the random vectors $\mfw_{\nu-N+1:\nu-1}$, the expectation would be marginalized out. Letting $\alpha\rightarrow +\infty$ would make $J_N$ tend to minus infinity.
Thus, unless \eqref{eq:psd_iff_cond} holds, irrespective of the initial condition $\bar{x}$ there is a planning horizon length realization $N \in \{2,\ldots, \nu\}$ such that the cost function is not bounded from below.

On the other hand, if  \eqref{eq:psd_iff_cond} holds but $\ker(\mfRtrue_{\nu-N+1})\not\subset \left[\ker(\mfStrue_{\nu-N+1}^T)\cap \ker(\gtrue_{\nu-N+1}^T)\right]$, then there exists a vector $v\in\mR^m$ with norm one, such that $\mfRtrue_{\nu-N+1} v=0$, and such that $\mfStrue_{\nu-N+1}^T v\neq 0$ or $\gtrue_{\nu-N+1}^T v\neq 0$.
Without loss of generality, assume that $\gtrue_{\nu-N+1}^Tv \geq 0$; otherwise we instead consider $-v$.
By Assumption \ref{ass:persistent_excitation}, it follows that we can find $\rho>0$ such that $\mP\left(\bar{\mfx}\in\mathscr{B}^{n}_\epsilon(\rho\mfS_{\nu-N+1}^Tv)\right)>0,\forall \epsilon>0$.
Now, consider $\mfu_{\nu-N+1}=\alpha v$ and $\mfx_{\nu-N+1}=\bar{x}=\rho\mfStrue_{\nu-N+1}^T v + \tilde{v}$, where $\tilde{v}\in \mathscr{B}^{n}_{\epsilon_1}(0)$ and where $\epsilon_1 > 0$ will be determined shortly. 
Note that with such choice, the expectation in \eqref{eq:JN_inequality} would be again marginalized out. Then it holds for the objective function that
\begin{align*}
&J_N\le \alpha[\underbrace{\rho v^T\mfStrue_{\nu-N+1}\mfStrue_{\nu-N+1}^T v+\gtrue_{\nu-N+1}^T v}_{= \rho\| \mfStrue_{\nu-N+1}^Tv \|^2 + \gtrue_{\nu-N+1}^Tv \; > \; 0}+\tilde{v}^T\mfStrue_{\nu-N+1}^T v]\\
&+(\rho v^T\mfStrue_{\nu-N+1}+\tilde{v}^T)\mfStrue_{\nu-N+1}^T\mfRtrue_{\nu-N+1}^\dagger\mfStrue_{\nu-N+1}\\
&\times (\rho \mfStrue_{\nu-N+1}^T v+\tilde{v}) + \gtrue_{\nu-N+1}^T\mfRtrue_{\nu-N+1}^\dagger \mfStrue_{\nu-N+1}\\
&\times  (\rho \mfStrue_{\nu-N+1}^T v+\tilde{v}) + \tau(\rho\mfStrue_{\nu-N+1}^T v+\tilde{v}).
\end{align*}
In particular, we can always make $\epsilon_1 >0$ small enough so that $\rho\| \mfStrue_{\nu-N+1}^Tv \|^2 + \gtrue_{\nu-N+1}^Tv+\tilde{v}^T\mfStrue_{\nu-N+1}^T v>0$, $\forall \tilde{v} \in  \mathscr{B}^{n}_{\epsilon_1}(0)$.
By taking $\alpha \to - \infty$, the upper bound goes to minus infinity and hence pushes $J_N$ to minus infinity. 
Recalling that $\mP(\bar{\mfx}\in\mathscr{B}^{n}_\epsilon(\rho\mfS_{\nu-N+1}^Tv))>0,\forall \epsilon\in(0,\epsilon_1)$, this shows that there exists a set of initial value realizations $\bar{x}$ with non-zero probability such that the forward problem is ill-posed.
This proves that $\neg$2)$\implies \neg$1), and thus that 1)$\implies$2).

Next, we prove that 2)$\implies$4). 
We make an ansatz that the solution to the HJBE \eqref{eq:stochastic_dyn_prog} is $V_t(\chi_t)=\frac{1}{2}\chi_t^T\Ptrue_t\chi_t+\etatrue_t^T\chi_t+\gammatrue_t$, where $\Ptrue_{1:\nu}$, $\etatrue_{1:\nu}$ and $\gammatrue_{1:\nu}$ are generated by \eqref{eq:generalized_riccati_iterations} and \eqref{eq:stochastic_riccati_iteration_constant_term}, respectively. 
The ansatz fulfils \eqref{eq:stochastic_dyn_prog_terminal} and \eqref{eq:stochastic_riccati_iteration_constant_term}.
Plugging the ansatz into \eqref{eq:stochastic_dyn_prog_iteration}, we have that
\begin{align*}
&\frac{1}{2}\chi_t^T\Ptrue_t\chi_t+\etatrue_t^T\chi_t+\gammatrue_t =\min_{\mu_t}\Big\{\frac{1}{2}\chi_t^T\Qtrue \chi_t+\qtrue^T\chi_t \nonumber\\
&+\frac{1}{2}\mu_t^T\Rtrue \mu_t+\mE_{\mfw_t}\Big[\frac{1}{2}(A\chi_t+B\mu_t+d+\mfw_t)^T\Ptrue_{t+1}\nonumber\\
&\times (A\chi_t+B\mu_t+d+\mfw_t)+\etatrue_{t+1}^T(A\chi_t+B\mu_t+d+\mfw_t),\nonumber\\ 
&+\gammatrue_{t+1}\Big]\Big\},\quad t=1:\nu-1.
\end{align*}
By Assumption \ref{ass:IID}, we can expand the expectation regarding $\mfw_t$ on the right hand side of the above equation, and removing constant terms that are irrelevant to the optimization we get
\begin{align}
&\min_{\mu_t}\Big\{\frac{1}{2}\mu_t^T\mfRtrue_t\mu_t+(\mfStrue_t\chi_t+\gtrue_t)^T\mu_t+(A\chi_t+d)^T\Ptrue_{t+1}\nonumber\\
&\times \! (A\chi_t \! + \! d) \! + \! \eta_{t+1}^T(A\chi_t \! + \! d) \!+ \! \frac{1}{2}\chi_t^T\Qtrue \chi_t \! + \! \qtrue^T\chi_t\Big\}.\label{eq:plug_in_HJBE}
\end{align}
\eqref{eq:plug_in_HJBE} is an unconstrained quadratic optimization problem with respect to $\mu_t$. 
By \eqref{eq:psd_iff_cond} we know that it is a convex problem, and hence it has an optimal solution if and only if the gradient is zero in some point (see, e.g., \citep[Prop.~1.1.2]{bertsekas1999nonlinear})
To verify that it has a solution, and to work out the optimal control, we take the derivative of \eqref{eq:plug_in_HJBE} with respect to $\mu_t$ and equate it to zero, which gives
\[
\mfRtrue_t \mu_t = -\mfStrue_t \chi_t -\gtrue_t, \quad t=1:\nu-1.
\]
Since \eqref{eq:kernel_containment} holds, we have that $\mfStrue_t \chi_t +\gtrue_t \in \im(\mfStrue_t) \oplus \im(\gtrue_t) =   \Big[\ker(\mfStrue_t^T)\cap \ker(\gtrue_t^T)\Big]^\perp \subset \ker(\mfRtrue_t)^\perp = \im(\mfRtrue_t)$ (see, e.g., \citep[Prop.~2.3]{engl2000regularization} for the last equality), where $\oplus$ denotes the direct sum of the subspaces.
Hence, the above equation has a solution, and the control signal takes the form of 
\begin{align}
\mu_t = -\mfRtrue_t^\dagger(\mfStrue_t \chi_t +\gtrue_t)+\mathcal{P}^{\ker(\mfRtrue)}_t \lambda_t, \forall\lambda_t\in\mathbb{R}^n, t=1:\nu-1,
\label{eq:opt_ctrl}
\end{align}
which minimizes the right hand side of \eqref{eq:plug_in_HJBE}.
Plug \eqref{eq:opt_ctrl} into \eqref{eq:plug_in_HJBE}, use the property of \eqref{eq:kernel_containment} and in view of \eqref{eq:generalized_riccati_iterations}, \eqref{eq:stochastic_riccati_iteration_constant_term}, the quadratic, first order and constant terms regarding $\chi_t$ equates between the left and right hand sides. Thus the ansatz is indeed a solution to the HJBE.

So far, we have shown 4)$\implies$1)$\implies$2)$\implies$4), i.e., the equivalence among 1), 2) and 4). To complete the proof, we now show the equivalence between 2) and 3).

First, we prove 2)$\implies$ 3), and start with noting that if \eqref{eq:generalized_riccati_iterations} holds, then \eqref{eq:nes_suff_existence_solution_1} holds trivially.
Next, with $\xitrue_t=\gtrue_t^T\mfRtrue_t^\dagger \gtrue_t$, we know from the above argument that due to the kernel containment \eqref{eq:kernel_containment}, $\Htrue_t$ can be expressed as
\begin{align*}
&\Htrue_t = \begin{bmatrix}
\mfRtrue_t &\mfStrue_t &\gtrue_t\\
\mfStrue_t^T &\mfStrue_t^T\mfRtrue_t^\dagger \mfStrue_t &\mfStrue_t^T\mfRtrue_t^\dagger \gtrue_t\\
\gtrue_t^T & \gtrue_t^T\mfRtrue_t^\dagger\mfStrue_t&\gtrue_t^T\mfRtrue_t^\dagger \gtrue_t
\end{bmatrix}\\
&=\begin{bmatrix}
\mfRtrue_t \\\mfStrue_t^T \\\gtrue_t^T
\end{bmatrix}
\mfRtrue_t^\dagger
\begin{bmatrix}
\mfRtrue_t& \mfStrue_t&\gtrue_t
\end{bmatrix},\quad t=1:\nu-1.
\end{align*}
By \eqref{eq:psd_iff_cond}, $\mfRtrue_t^\dagger\succeq 0$ and hence $\Htrue_t\succeq 0$, i.e., \eqref{eq:nes_suff_existence_solution_2} holds.
On the other hand, by the rank property of Schur complement \citep[p.~43]{horn2005basic}, it holds that
\[
\rank(\Htrue_t) =\rank(\mfRtrue_t)+\rank(\Htrue_t \backslash \mfRtrue_t ).
\]
Now, observe that
\begin{align*}
&\Htrue_t \backslash \mfRtrue_t  = 
 \begin{bmatrix}
\mfStrue_t^T\mfRtrue_t^\dagger \mfStrue_t &\mfStrue_t^T\mfRtrue_t^\dagger \gtrue_t\\
\gtrue_t^T\mfRtrue_t^\dagger\mfStrue_t&\gtrue_t^T\mfRtrue_t^\dagger \gtrue_t
\end{bmatrix}
-
\begin{bmatrix}
\mfStrue_t^T\\ \gtrue_t^T
\end{bmatrix}
\mfRtrue_t^\dagger
\begin{bmatrix}
\mfStrue_t &\gtrue_t
\end{bmatrix}
\\
 &= \begin{bmatrix}
\mfStrue_t^T\mfRtrue_t^\dagger \mfStrue_t &\mfStrue_t^T\mfRtrue_t^\dagger \gtrue_t\\
\gtrue_t^T\mfRtrue_t^\dagger\mfStrue_t&\gtrue_t^T\mfRtrue_t^\dagger \gtrue_t
\end{bmatrix}
-
 \begin{bmatrix}
\mfStrue_t^T\mfRtrue_t^\dagger\mfStrue_t &\mfStrue_t^T\mfRtrue_t^\dagger \gtrue_t\\
\gtrue_t^T\mfRtrue_t^\dagger\mfStrue_t &\gtrue_t^T\mfRtrue_t^\dagger \gtrue_t
\end{bmatrix} = 0,
\end{align*}
and hence $\rank(\Htrue_t) =\rank(\mfRtrue_t)$, i.e., \eqref{eq:nes_suff_existence_solution_3} holds.

Now we prove 3)$\implies$2).
If \eqref{eq:nes_suff_existence_solution} holds, it follows that $\mfRtrue_t\succeq 0$, i.e., \eqref{eq:psd_iff_cond} holds.
By properties of the generalized Schur complement \citep[Thm.~1.20 and p.~43]{horn2005basic}, it follows that
\begin{align*}
&(I-\mfRtrue_t\mfRtrue_t^\dagger)
\begin{bmatrix}
\mfStrue_t &\gtrue_t
\end{bmatrix}=0, \implies\eqref{eq:kernel_containment},\\
&\Htrue_t\backslash \mfRtrue_t:=\begin{bmatrix}
A^T\Ptrue_{t+1}A+\Qtrue-\Ptrue_t &\betatrue_t\\
\beta_t &\xitrue_t
\end{bmatrix}
-\begin{bmatrix}
\mfStrue_t^T\\ \gtrue_t^T
\end{bmatrix}
\mfRtrue_t^\dagger
\begin{bmatrix}
\mfStrue_t &\gtrue_t
\end{bmatrix}\\
&=\begin{bmatrix}
A^T\Ptrue_{t+1}A+\Qtrue-\Ptrue_t &\betatrue_t\\
\betatrue_t &\xitrue_t
\end{bmatrix}-
\begin{bmatrix}
\mfStrue_t^T\mfRtrue_t^\dagger\mfStrue_t &\mfStrue_t^T\mfRtrue_t^\dagger \gtrue_t\\
\gtrue_t^T\mfRtrue_t^\dagger\mfStrue_t &\gtrue_t^T\mfRtrue_t^\dagger \gtrue_t
\end{bmatrix}\succeq 0,\\
&\rank(\Htrue_t)=\rank(\mfRtrue_t)+\rank(\Htrue_t\backslash \mfRtrue_t).
\end{align*}
Since $\rank(\Htrue_t)=\rank(\mfRtrue_t)$, we must have that $\Htrue_t\backslash \mfRtrue_t=0$, and hence \eqref{eq:generalized_riccati_iterations} follows.
This finishes the proof of the entire theorem.
\end{proof}

Before we continue, we make a few remarks related to Theorem~\ref{thm:indefinite_LQR} and the proof.

\begin{remark}\label{rem:rank_psd_and_riccati}
From the above proof, it is interesting to note that in the equivalence between point 2) and 3) in Theorem~\ref{thm:indefinite_LQR}, we have that \eqref{eq:nes_suff_existence_solution_1} and \eqref{eq:nes_suff_existence_solution_3} are equivalent to that the Riccati recursions in \eqref{eq:generalized_riccati_iterations} are satisfied, and that  \eqref{eq:nes_suff_existence_solution_2}, i.e., that $\Htrue_t \succeq 0$ holds, is equivalent to \eqref{eq:psd_iff_cond} and \eqref{eq:kernel_containment}.
\end{remark}

\begin{remark}\label{rem:state_space_expansion}
By extending the state space to $\tilde{\mfx}_t = [\mfx_t^T, 1]^T$ and with state space matrices given by
\begin{subequations}\label{eq:extended_state_space}
\begin{equation}\label{eq:extended_state_space_mtx}
\tilde{A} = \begin{bmatrix}
A &  d \\  \zero^T & 1 
\end{bmatrix}, \quad
\tilde{B} = \begin{bmatrix} B \\ 0 \end{bmatrix},
\end{equation}
and cost matrices given by
\begin{equation}\label{eq:extended_cost_mtx}
\tilde{\Qtrue} = \begin{bmatrix}
\Qtrue & \qtrue\\  
\qtrue^T & 0
\end{bmatrix}, \quad \tilde{\Rtrue} = \Rtrue,
\end{equation}
\end{subequations}
we can indeed rewrite the forward problem \eqref{eq:stochastic_forward_problem} as a ``standard" LQ problem (albeit possibly indefinite):
\begin{subequations}\label{eq:forward_problem_tilde}
  \begin{align}
   \min_{\substack{\mfx_{1:\nu} , \; \mfu_{1:\nu}}}& \; J_N := \mE_{\mfw_{\nu-N+1:\nu-1}} \Big[\frac{1}{2} \tilde{\mfx}_{\nu}^{T} \tilde{\Qtrue} \tilde{\mfx}_{\nu} \nonumber\\
   &+\sum_{t = \nu - N+1}^{\nu-1} [\frac{1}{2} \tilde{\mfx}_{t}^{T} \tilde{\Qtrue} \tilde{\mfx}_{t} + \frac{1}{2}\mfu_{t}^T\tilde{\Rtrue}\mfu_t] \Big] \label{eq:forward_problem_tilde_cost} \\
  \text{s.t.}
  & \;\; \tilde{\mfx}_{t+1} = \tilde{A}\tilde{\mfx}_{t} + \tilde{B}\mfu_{t} + \tilde{\mfw}_t, \nonumber \\
  &\;\;\;\; t = \nu-N+1, \ldots, \nu-1, \label{eq:forward_problem_tilde_dynamics} \\
  & \;\; \tilde{\mfx}_{t+1} = \tilde{\mfx}_t, \; t=1,\ldots,\nu-N\\
  & \;\; \tilde{\mfx}_{1} = \bar{\tilde{x}}, \label{eq:forward_problem_tilde_init_cond}\\
  & \;\; \mfu_1 = \cdots = \mfu_{\nu-N}=0,\label{eq:forward_problem_tilde_control_stand_still}
  \end{align}
\end{subequations}
with initial value $\bar{\tilde{x}} = [\bar{x}^{T}, 1]^T$, and where $\tilde{\mfw}_t = [\mfw_t^T, 0]^T$.
However, note that \eqref{eq:forward_problem_tilde} is not a classic LQR in the sense that $(\tilde{A}, \tilde{B})$ is not controllable. Moreover, $\bar{\tilde{\mfx}}$ does not satisfy Assumption~\ref{ass:persistent_excitation}. Furthermore, extending \eqref{eq:forward_problem_tilde} to time-varying $d_t$ or the tracking problem, where $q_t = Qx_t^r$ and $x_t^r$ is the reference signal, is not possible without specifying the problem structure \eqref{eq:extended_state_space}. Therefore, existing IOC results cannot be directly applied to \eqref{eq:forward_problem_tilde}.
\end{remark}

\subsection{Analysis of the closed-loop system matrices}

In view of \eqref{eq:optimal_ctrl_formula}, it seems like there might be infinitely many choices of control signals that are optimal. However, as we shall see, under the assumptions we impose the optimal control signal is \emph{unique} for the considered forward problem \eqref{eq:stochastic_forward_problem}.
\begin{proposition}\label{prop:mfR_pd}
Let $\Rtrue \succ 0$. Under Assumptions~\ref{ass:controlability_and_full_rank}, \ref{ass:IID} and \ref{ass:persistent_excitation}, $(\Qtrue, \qtrue) \in \mscrF(\Rtrue)$ is equivalent to
$\mfRtrue_t \succ 0$, $t=1:\nu-1$.
\end{proposition}

\begin{proof}
The implication ``$\Longleftarrow$'' follows from Theorem~\ref{thm:indefinite_LQR}. To prove the implications ``$\Longrightarrow$'':
by Theorem~\ref{thm:indefinite_LQR}, since $(\Qtrue,\qtrue)\in\mathscr{F}(\Rtrue)$, there exist matrices and vectors that fulfil \eqref{eq:generalized_riccati_iterations} and \eqref{eq:psd_kernel_containment_cond}. In particular, since $A$ is invertible, by \eqref{eq:mfS_true} and \eqref{eq:mfR_def} we have that
\[
\mfRtrue_t = B^T \Ptrue_{t+1}B + \Rtrue = B^T \Ptrue_{t+1}AA^{-1}B + \Rtrue = \mfStrue_tA^{-1}B + \Rtrue,
\]
for $t=1:\nu-1$.
Now, by  \eqref{eq:kernel_containment} we have that $\ker(\mfRtrue_t) \subset \ker(\mfStrue_t^T)$ holds for $t=1:\nu-1$. In particular, this means that
\[
z \in \ker(\mfRtrue_t) \; \Longrightarrow \;
\begin{cases}
\mfRtrue_t z = 0\\
\mfStrue_t^T z = 0
\end{cases}
\!\! \Longrightarrow \;
\begin{cases}
-\mfStrue_tA^{-1}B z = \Rtrue z \\
z^T \mfStrue_t = 0.
\end{cases}
\]
This in turn means that for any $z \in \ker(\mfRtrue_t) $,
\[
z^T \Rtrue z = - z^T \mfStrue_t A^{-1}B z = 0,
\]
and since $\Rtrue \succ 0$ this means that $z = 0$. Therefore, the only vector in $\ker(\mfRtrue_t)$ is the zero-vector, and since $\mfRtrue_t$ is positive semi-definite (see \eqref{eq:psd_iff_cond}), this implies that $\mfRtrue_t$ is in fact strictly positive definite for $t=1:\nu-1$.
\end{proof}

\begin{corollary}\label{corr:no_nullspace_component}
Under the assumptions in Proposition~\ref{prop:mfR_pd}, the optimal control signal for problem \eqref{eq:stochastic_forward_problem} takes the form
\begin{subequations}\label{eq:optimal_ctrl_formula_no_kernel}
\begin{align}
\mfu_t & = -\mfRtrue_t^{-1}(\mfStrue_t \mfx_t +\gtrue_t), \; t=\nu-N+1:\nu-1 \\
\mfu_t &= 0, \quad t=1:\nu-N.
\end{align}
\end{subequations}
\end{corollary}

\begin{remark}
Even if $\Rtrue$ is not strictly positive definite, by the proof of Proposition~\ref{prop:mfR_pd} we can still partially characterize $\ker(\mfRtrue_t)$. In particular, $z \in \ker(\mfRtrue_t)$ implies that $z^T \Rtrue z = 0$. If $\Rtrue$ is full rank, this is the neutral subspace in the indefinite inner product space defined by $\Rtrue$; see, e.g., \citep[Chp.~2]{gohberg2005indefinite}.
\end{remark}

By Corollary~\ref{corr:no_nullspace_component}, under the conditions in Proposition~\ref{prop:mfR_pd} there is a unique solution to the forward optimal control problem \eqref{eq:stochastic_forward_problem}.
In this case, the system's behavior (after it starts applying control) is determined by $\tilde{\mfx}_{t+1} = \tilde{A}_{cl}(t;\Qtrue,\qtrue) \tilde{\mfx}_t + \tilde{\mfw}_t$, for $t = \nu - N + 1 : \nu$, where  $\tilde{A}_{cl}(t;\Qtrue,\qtrue)$ is the closed-loop system matrix at time $t$ for the extended state-space model (see Remark~\ref{rem:state_space_expansion}).
In particular,
\begin{equation}\label{eq:tildeAcl}
\tilde{A}_{cl}(t;\Qtrue,\qtrue) = 
\begin{bmatrix}
A - B\mfRtrue_t^{-1} \mfStrue_t & d - B\mfRtrue_t^{-1} \gtrue_t  \\
\mathbf{0}^T & 1
\end{bmatrix},
\end{equation}
with $\mfRtrue_t$, $\mfStrue_t$, and $\gtrue_t$ as in \eqref{eq:psd_kernel_containment_cond}, and hence it implicitly depends on $(\Qtrue,\qtrue)$ in the objective function \eqref{eq:stochastic_forward_problem_cost}.
This means that the (conditional) distribution of the agent's optimal trajectory $\mP(\mfx_{1:\nu}\mid \mfN=N,\bar{\mfx}=\bar{x})$ and optimal control $\mP(\mfu_{1:\nu-1} \mid \mfN=N,\bar{\mfx}=\bar{x})$ are implicitly given by solving \eqref{eq:stochastic_forward_problem}.
Moreover, under mild regularity conditions on the probability distribution of $(\bar{\mfx}, \mfN)$, the formulation in \eqref{eq:stochastic_forward_problem} then defines joint probability distributions for $(\mfx_{1:\nu},\mfN,\bar{\mfx})$ and $(\mfu_{1:\nu-1},\mfN,\bar{\mfx})$ (cf.~\citep[Prop.~1.15]{shao1999mathematical} or \citep[Thm.~5.3]{kallenberg1997foundations}).
Before we continue analyzing the identifiability, we present the following corollary that is useful in the analysis to come.

\begin{corollary}\label{cor:Acl_full_rank}
Under the assumptions in Proposition~\ref{prop:mfR_pd}, 
given any $R\succ 0$, for any $(Q, q) \in \mscrF(R)$, let $\{ P_t \}_ {t=1}^{\nu}$ be the solution to \eqref{eq:generalized_riccati_iterations} that corresponds to $Q$.
Accordingly, let $\mfR_t$, $\mfS_t$ and $g_t$ be defined as in Theorem~\ref{thm:indefinite_LQR}, for $t = 1:\nu-1$.
Then the matrix
\[
\Acl(t;Q) := A - B \mfR_t^{-1}\mfS_t,
\]
as well as the matrix $\tilde{A}_{cl}(t; Q, q)$ in \eqref{eq:tildeAcl}, are invertible for all $t = 1:\nu-1$.
\end{corollary}

\begin{proof}
The fact that the matrix $\Acl(t;Q)$ is invertible for $t = 1:\nu-1$ follows by an argument similar to the proof of \citep[Thm.~2.1]{zhang2019inverse}, since $R \succ 0$ and $\mfR_t \succ 0$ holds for $t = 1:\nu-1$.
To show that  $\tilde{A}_{cl}(t; Q, q)$ has full rank, we simply note that it has the upper block-triangular form \eqref{eq:tildeAcl}, and since $\Acl(t;Q)$ has full rank so does $\tilde{A}_{cl}(t; Q, q)$.
\end{proof}

\section{Identifiability analysis and persistent excitation}\label{sec:identifiability}
Next, we investigate the inverse problem of recovering $(\Qtrue, \qtrue)$ from observations of optimal trajectories to problem \eqref{eq:stochastic_forward_problem}. Throughout the rest we will therefore, unless explicitly stated otherwise, assume that $\Rtrue = I$ and $(\Qtrue,\qtrue)\in\mscrF(\Rtrue=I)$.
We start by considering the identifiability of the problem. 

To this end, first note that from the analysis in Section~\ref{sec:forward_problem}, for any parameters $(Q,q)\in\mscrF(I)$, the agent's behavior is completely determined by the time-varying closed-loop system matrices $\tilde{A}_{cl}(t; Q, q)$ in \eqref{eq:tildeAcl}.
In the spirit of \citep[Sec.~5.1]{ljung1994global}, we can thus see the sequence of closed-loop system matrices $\{ \tilde{A}_{cl}(t; Q, q) \}_{t = 1}^{\nu-1}$ as the model structure.
 Therefore, the fundamental question for identifiability is if there exist two different sets of parameters $(Q, q)$ and $(Q^\prime, q^\prime)$ such that  $\tilde{A}_{cl}(t; Q, q)  = \tilde{A}_{cl}(t; Q^\prime, q^\prime)$ for all $t = 1:\nu-1$.

\begin{proposition}[Identifiability]\label{prop:structural_identifiability}
Under Assumptions~\ref{ass:controlability_and_full_rank}, \ref{ass:IID}, \ref{ass:planning_horizon} and \ref{ass:persistent_excitation}, given  $(Q, q), (Q^\prime, q^\prime) \in \mscrF(I)$, if $\tilde{A}_{cl}(t; Q, q)  = \tilde{A}_{cl}(t; Q^\prime, q^\prime)$ for all $t = 1:\nu-1$, then  $(Q, q) = (Q^\prime, q^\prime)$.
\end{proposition}

\begin{proof}
Assume that $\tilde{A}_{cl}(t;Q,q)=\tilde{A}_{cl}(t;Q^\prime,q^\prime)$ for $t=1:\nu-1$, and let $(P_k, \eta_k, \mfS_k, \mfR_k, \g_k)_{k=\nu-N+1}^{\nu-1}$ and $(P_k^\prime, \eta_k^\prime, \mfS_k^\prime, \mfR_k^\prime, \g_k^\prime)_{k=\nu-N+1}^{\nu-1}$ be the solutions to \eqref{eq:generalized_riccati_iterations} and \eqref{eq:psd_kernel_containment_cond} for $(Q, q)$ and $(Q^\prime, q^\prime)$, respectively. Moreover, let $Q^\prime = Q+\Delta Q$, $q^\prime = q + \Delta q$, $P_t^\prime = P_t+\Delta P_t$, $\eta_t^\prime = \eta_t+\Delta \eta_t$, $\mfS_t^\prime = \mfS_t + \Delta \mfS_t$, $\mfR_t^\prime = \mfR_t + \Delta \mfR_t$, and $\g_t^\prime = \g_t + \Delta \g_t$.
Since $\tilde{A}_{cl}(t;Q,q)=\tilde{A}_{cl}(t;Q^\prime,q^\prime)$, for $t=1:\nu-1$, it follows that $A_{cl}(t;Q)=A_{cl}(t;Q^\prime)$, for $t=1:\nu-1$.
Then following the line of arguments in \citep[Thm.~2.1]{zhang2019inverse}, we can conclude that $\Delta Q = 0$. This in turn implies that $\Delta P_t$, $\Delta \mfR_t$ and $\Delta \mfS_t$ are all zero (cf.~\eqref{eq:generalized_riccati_iterations} and \eqref{eq:psd_kernel_containment_cond}).

Next, since $\tilde{A}_{cl}(t;Q,q)=\tilde{A}_{cl}(t;Q^\prime,q^\prime)$ and $\Delta \mfR_t = 0$, for $t=1:\nu-1$, it follows that
$d - B\mfR_t^{-1} \g_t = d - B\mfR_t^{-1} \g_t^\prime$ for $t=1:\nu-1$. Therefore, it holds that $ B\mfR_t^{-1} \Delta \g_t = 0$ for $t=1:\nu-1$. Since $B$ is full column rank by assumption, and since $\mfR_t^{-1} \succ 0$ by Proposition~\ref{prop:mfR_pd}, we must have that $\Delta \g_t = 0$ for $t=1:\nu-1$.
In view of \eqref{eq:mfg_true}, we therefore have that
\[
B^T \Delta \eta_{t+1} =  \Delta g_t  = 0, \quad t=1:\nu-1.
\]
Next, in view of \eqref{eq:generalized_riccati_3} and \eqref{eq:generalized_riccati_4}, this in turn implies that
\begin{align*}
\Delta \eta_\nu = \Delta q, \quad \Delta\eta_t = A^T\Delta\eta_{t+1}+\Delta q,\quad t=1:\nu-1.
\end{align*}
Thus, we have
\begin{subequations}\label{eq:Delta q}
\begin{align}
&\Delta \eta_\nu = \Delta q,\quad B^T \Delta \eta_\nu=B^T\Delta q = 0,\\
&B^T\Delta \eta_{\nu-1}=B^T(A^T\Delta \eta_\nu+\Delta q)\nonumber\\
&=B^TA^T\Delta q+\underbrace{B^T\Delta q}_{=0} = 0.
\end{align}
Furthermore,
\begin{align}
&B^T\Delta\eta_{\nu-2} = B^T(A^T\Delta\eta_{\nu-1}+\Delta q)\nonumber\\
&=B^T((A^T)^2\Delta \eta_\nu+A^T\Delta q+\Delta q)\nonumber\\
&=B^T(A^T)^2\Delta q+\underbrace{B^TA^T\Delta q+B^T\Delta q}_{=0}=0,\\
&\ldots\nonumber\\
&B^T\eta_{1} = B^T(A^T)^{\nu-1}\Delta q \nonumber\\
&\qquad+\underbrace{B^T(A^T)^{\nu-2}\Delta q+\ldots+B^T\Delta q}_{=0}=0.
\end{align}
\end{subequations}
Writing \eqref{eq:Delta q} in a compact form gives
\[
\underbrace{\begin{bmatrix}
B^T\\B^TA^T\\\vdots\\B^T(A^T)^{\nu-1}
\end{bmatrix}}_{\Gamma}\Delta q = 0.
\]
Since $(A,B)$ is controllable by Assumption~\ref{ass:controlability_and_full_rank}, and $\nu \geq n+1$ by Assumption~\ref{ass:planning_horizon}, the matrix  $\Gamma$ has full column rank, and hence $\Delta q=0$. The fact that $\Delta Q=0, \Delta q=0$ implies that $Q=Q^\prime, q=q^\prime$.
\end{proof}

This means that the parameters $(Q,q)$ that characterizes the closed-loop system matrices are identifiable.
Moreover, in view of \eqref{eq:stochastic_forward_problem}, we can see $(\bar{\mfx},\mfN)$ as the ``input" of the model and $\mfy_{1:\nu}$ as the ``output".
To this end, in order to uniquely identify the parameters $(\Qtrue,\qtrue)$, the ``input" $(\bar{\mfx},\mfN)$ needs to be ``persistently exciting" \citep[Sec.~5.1]{ljung1994global}. Notably, Assumption \ref{ass:planning_horizon} gives the persistent excitation assumption regarding $\mfN$.
Moreover, Assumption~\ref{ass:persistent_excitation} turns out to give a persistent excitation condition for the initial value $\bar{\mfx}$.
In fact, we have the following result, the proof of which we defer to the appendix. 

\begin{lemma}\label{lem:positive_definite_cov}
Let $(\bar{\mfx},\mfN)$ be as in Assumption~\ref{ass:persistent_excitation}. Then, for all $N \in \{2, \ldots, \nu\}$ such that $\mP(\mfN = N) > 0$,  $\cov_{\bar{\mfx} \mid \mfN = N}(\bar{\mfx},\bar{\mfx})\succ 0$.
\end{lemma}

\begin{proof}
See Appendix.
\end{proof}

This result can now be used to prove the following Lemma, which is useful in the IOC algorithm construction to come. Similarly, the proof of this Lemma is also deferred to the appendix.

\begin{lemma}[Persistent excitation]\label{lem:stochastic_persistent_excitation}
Suppose that $(\Qtrue,\qtrue)\in\mscrF(I)$ and let $\tilde{\mfx}_t:=[\mfx_t^T,1]^T$.
Under Assumptions~\ref{ass:controlability_and_full_rank}, \ref{ass:IID}, \ref{ass:planning_horizon}, and \ref{ass:persistent_excitation}, it holds that $\mE_{\mfx_t|\mfN=\nu}[\tilde{\mfx}_t\tilde{\mfx}_t^T]\succ 0$, and $\mE[\|\tilde{\mfx}_t\|^2]<\infty$ for all $ t=1:\nu$.
\end{lemma}
\begin{proof}
See Appendix.
\end{proof}

\section{The IOC algorithm}\label{sec:IOC_algorithm}
In this section, we construct the IOC algorithm for general linear-quadratic systems with different time-horizon lengths.
In particular, we show that the algorithm is statistically consistent, i.e., that it converges in probability to the true underlying parameter.
For the sake of brevity, in some of the following we sometimes use the notation $(\cdot)$ for the arguments of some functions.

In order to construct the IOC algorithm, we further make the following assumption.
\begin{assumption}[Bounded parameters]\label{ass:bounded_parameter}
The parameter tuple $(\Qtrue,\qtrue)$ that governs the agents tracking behaviour lies in the compact set
\[
\mathbb{G}(\varphi) :=\left\{(\Qtrue\in\mS^n,\qtrue\in\mR^n) \mid \| 
\begin{bmatrix}
\Qtrue &\qtrue\\
\qtrue &0
\end{bmatrix}\|_F\le \varphi\right\},
\]
for some (potentially unknown) $0 < \varphi < \infty$.
\end{assumption}
This assumption is mild, since when we solve the corresponding inverse problem in practice, we can always set $\varphi$ arbitrary large if we have no prior knowledge on the norm bound of the parameters.

\subsection{Construction and empirical approximation}\label{subsec:construction_and_approximation}
To this end, the algorithm is constructed based on the necessary and sufficient optimality conditions in Theorem \ref{thm:indefinite_LQR}.
More precisely, we are interested in finding the best $(Q^\star,q^\star)\in\mscrF(I)$ that suits the collected optimal state trajectory observations $\{y_t^i\}_{t=1}^{\nu}$ of the agent's trials $i=1:M$.
Thus, the IOC algorithm will be built upon an optimization problem
which is constructed so that it has a unique optimal solution $(Q^\star,q^\star)$ which is the ``true" $(\Qtrue,\qtrue)\in \mscrF(I)$.

First, we construct the objective function for the optimization problem. 
Given a realization of the planning horizon $N$, an optimal control signal $\mu_t^\star$ that is optimal to $(Q,q)$ (assuming that $(Q,q) \in \mathscr{F}(I)$) and an optimal control signal $\bar{\mu}_t$ that is optimal to $(\Qtrue,\qtrue)$, it holds for all state $\chi_t\in\mR^n$ that
\begin{align*}
0&=\frac{1}{2}\chi_t^TQ\chi_t\!+\!q^{T}\chi_t\!+\!\frac{1}{2}\|\mu_t^{\star}\|^2\!+\!\mE_{\mfw_t}[V_{t+1}(A\chi_t\!+\!B\mu_t^{\star}\!+\!d\\
&\quad \!+\!\mfw_t)]\!-\!V_t(\chi_t)\\
&=\min_{\mu_t}\Big\{\frac{1}{2}\chi_t^TQ\chi_t+q^{T}\chi_t+\frac{1}{2}\|\mu_t\|^2\\
&\quad +\mE_{\mfw_t}[V_{t+1}(A\chi_t+B\mu_t+d+\mfw_t)]\Big\}-V_t(\chi_t)\\
&\le  \frac{1}{2}\chi_t^TQ\chi_t\!+\!q^{T}\chi_t\!+\!\frac{1}{2}\|\bar{\mu}_t \|^2\!+\!\mE_{\mfw_t}[V_{t+1}(A\chi_t\!+\!B\bar{\mu}_t\!+\!d\\
&\quad \!+\!\mfw_t)]\!-\!V_t(\chi_t),
\end{align*}
where the inequality follows since $\bar{\mu}_t$ is not necessarily optimal to $(Q,q,\Rtrue=I)$, and where $V_t(\cdot)$ has the form \eqref{eq:Vt}, and $P_{t:t+1}$, $\eta_{t:t+1}$, and $\gamma_{t:t+1}$ in $V_t(\cdot)$ are determined by $(Q,q)$ via \eqref{eq:generalized_riccati_iterations}.
Moreover, seen intuitively from the other perspective, for given $\chi_t$ and $\bar{\mu}_t$, we expect the inequality to hold unless we plug in $(Q,q,\Rtrue=I)$ which renders the state $\chi_t$ and control $\bar{\mu}_t$ optimal.
We hence define the ``violation" of HJBE at time step $t$ by
\begin{align*}
&\psi_{t,N}(Q,q;\chi_t,\mu_t):=\\
&\mE_{\mfw_t}[ V_{t+1}(A\chi_t + B \mu_t + d+\mfw_t)] +\frac{1}{2} \chi_t^T Q \chi_t+q^T \chi_t \\
&\quad+ \frac{1}{2}\| \mu_t \|_2^2- V_t(\chi_t), \quad t=\nu-N+1:\nu-1,
\end{align*}
since we expect that $\psi_{t,N}(Q,q;\chi_t,\mu_t) \geq 0$. The latter will be formally proved in Theroem \ref{thm:IOC_Q_optimal}.

Now, plugging \eqref{eq:Vt} and \eqref{eq:stochastic_riccati_iteration_constant_term} in to the above equation, we have
\begin{align}
&\psi_{t,N}(Q,q;\chi_t,\mu_t)=\mE_{\mfw_t}\Big[\frac{1}{2}(A\chi_t+B\mu_t+d+\mfw_t)^TP_{t+1} \nonumber \\
&\times(A\chi_t+B\mu_t+d+\mfw_t)+\eta_{t+1}^T(A\chi_t+B\mu_t+d+\mfw_t)\Big] \nonumber \\
&+\frac{1}{2}\chi_t^TQ\chi_t+q^T\chi_t+\frac{1}{2}\|\mu_t\|^2-\frac{1}{2}\chi_t^TP_t\chi_t-\eta_t^T\chi_t \nonumber \\
&+\frac{1}{2}g_t^T\mfR_t^\dagger g_t-\frac{1}{2}d^TP_{t+1}d-\eta_t^Td-\frac{1}{2}\tr(P_{t+1}\Sigma_w). \label{eq:psi_t_N}
\end{align}
Given a realization of the planning horizon $N$, let $\mfx_{\nu-N+1:\nu}$ and $\mfu_{\nu-N+1:\nu-1}$ be the optimal trajectory and control. We let $\chi_t=\mfx_t$ and $\mu_t=\mfu_t$, and take the expectation of $\psi_{t,N}(Q,q;\mfx_t,\mfu_t)$ with respect to $\mfx_t|\mfN=N$.
In view of \eqref{eq:stochastic_forward_problem_dynamics}, this gives
\begin{align}
&\mE_{\mfx_t|\mfN=N}[\psi_{t,N}(Q,q;\mfx_t,\mfu_t)]=\nonumber\\
&\mE_{\mfx_t|\mfN=N}\Big[\mE_{\mfw_t}\Big[\frac{1}{2}\underbrace{(A\mfx_t+B\mfu_t+d+\mfw_t)}_{\mfx_{t+1}}^TP_{t+1}\nonumber\\
&\times\underbrace{(A\mfx_t+B\mfu_t+d+\mfw_t)}_{\mfx_{t+1}}+\eta_{t+1}^T\underbrace{(A\mfx_t+B\mfu_t+d+\mfw_t)}_{\mfx_{t+1}}\Big]\nonumber\\
&+\frac{1}{2}\mfx_t^TQ\mfx_t+q^T\mfx_t+\frac{1}{2}\|\mfu_t\|^2-\frac{1}{2}\mfx_t^TP_t\mfx_t-\eta_t^T\mfx_t\nonumber\\
&+\frac{1}{2}g_t^T\mfR_t^\dagger g_t-\frac{1}{2}d^TP_{t+1}d-\eta_t^Td-\frac{1}{2}\tr(P_{t+1}\Sigma_w)\Big]\nonumber\\
&=\mE_{\mfx_{t+1}|\mfN=N}\Big[\frac{1}{2}\mfx_{t+1}^TP_{t+1}\mfx_{t+1}+\eta_{t+1}^T\mfx_{t+1}\Big]\nonumber\\
&+\mE_{\mfx_t|\mfN=N}\Big[\frac{1}{2}\mfx_t^TQ\mfx_t+q^T\mfx_t+\frac{1}{2}\|\mfu_t\|^2-\frac{1}{2}\mfx_t^TP_t\mfx_t-\eta_t^T\mfx_t\Big]\nonumber\\
&+\frac{1}{2}g_t^T\mfR_t^\dagger g_t-\frac{1}{2}d^TP_{t+1}d-\eta_t^Td-\frac{1}{2}\tr(P_{t+1}\Sigma_w).\label{eq:obj_construct_x_t}
\end{align}
However, the above expression is constructed based on $\mfx_{t}$, while the observations are in terms of $\mfy_t$.
To rewrite it in terms of $\mfy_t$, first we simply add and subtract some terms in the expression above:
\begin{align}
&\mE_{\mfx_t|\mfN=N}[\psi_{t,N}(Q,q;\mfx_t,\mfu_t)]= \mE_{\mfx_t|\mfN=N}[\psi_{t,N}(Q,q;\mfx_t,\mfu_t)] \nonumber \\
& +\frac{1}{2}\tr(P_{t+1}\Sigma_v)-\frac{1}{2}\tr(P_t\Sigma_v) + \frac{1}{2} \tr(Q \Sigma_v) \nonumber \\
& -\frac{1}{2}\tr(P_{t+1}\Sigma_v)+\frac{1}{2}\tr(P_t\Sigma_v) - \frac{1}{2} \tr(Q \Sigma_v). \label{eq:obj_construct_add_and_subtract}
\end{align}
On the other hand, by Assumption~\ref{ass:IID},
$\{\rvv_t\}_{t=1}^{\infty}$ are independent of any other stochastic elements.  Using the cyclic permutation property of the matrix trace operator, we know that
\begin{align*}
&\mE_{\rvv_t}[\rvv_t^TP_{t+1}\rvv_t]= \mE_{\rvv_t}[\tr(\rvv_t^TP_{t+1}\rvv_t)]=  \mE_{\rvv_t}[\tr(P_{t+1}\rvv_t\rvv_t^T)] \\
&=\tr(P_{t+1}\Sigma_v).
\end{align*}
Similarly, we also have $\mE_{\rvv_t}[\rvv_t^TP_t\rvv_t]=\tr(P_t\Sigma_v)$, $\mE_{\rvv_{t+1}}[\rvv_t^TQ\rvv_t]=\tr(Q\Sigma_v)$, $d^TP_{t+1}d=\tr(P_{t+1}dd^T)$.  In view of \eqref{eq:stochastic_riccati_iteration_constant_term}, \eqref{eq:noisy_observation} and the fact that $\mE_{\mfw_t}[\mfw_t]=0$, $\mE_{\rvv_t}[\rvv_t]=0$,
using \eqref{eq:obj_construct_x_t} and \eqref{eq:obj_construct_add_and_subtract} we can rewrite $\mE_{\mfx_t|\mfN=N}[\psi_{t,N}(Q,q;\mfx_t,\mfu_t)]$ as
\begin{align*}
&\mE_{\mfx_t|\mfN=N}[\psi_{t,N}(Q,q;\mfx_t,\mfu_t)] \\
&  = \mE_{\rvv_{t:t+1}}\Big[\mE_{\mfx_{t+1}|\mfN=N}\Big[\frac{1}{2}\underbrace{(\mfx_{t+1}+\rvv_{t+1})^T}_{\mfy_{t+1}^T}\\
& \; \times P_{t+1}\underbrace{(\mfx_{t+1}+\rvv_{t+1})}_{\mfy_{t+1}}+\eta_{t+1}^T\underbrace{(\mfx_{t+1}+\rvv_{t+1})}_{\mfy_{t+1}}\Big]\\
& \; +\mE_{\mfx_t|\mfN=N}\Big[\frac{1}{2}\underbrace{(\mfx_t+\rvv_t)^T}_{\mfy_t^T}Q\underbrace{(\mfx_t+\rvv_t)}_{\mfy_t}+q^T\underbrace{(\mfx_t+\rvv_t)}_{\mfy_t}\\
& \; +\frac{1}{2}\|\mfu_t\|^2-\frac{1}{2}\underbrace{(\mfx_t+\rvv_t)^T}_{\mfy_t^T}P_t\underbrace{(\mfx_t+\rvv_t)}_{\mfy_t}-\eta_t^T\underbrace{(\mfx_t+\rvv_t)}_{\mfy_t}\Big]\\
& \; +\frac{1}{2}g_t^T\mfR_t^\dagger g_t-\frac{1}{2}\tr(P_{t+1}dd^T)-\eta_t^Td-\frac{1}{2}\tr(P_{t+1}\Sigma_w)\\
& \; -\frac{1}{2}\tr(P_{t+1}\Sigma_v)+\frac{1}{2}\tr(P_t\Sigma_v) - \frac{1}{2} \tr(Q \Sigma_v)\Big] \\
&  = \!\!\! \mE_{\mfy_{t+1}|\mfN=N} \! \Big[\frac{1}{2}\mfy_{t+1}^T P_{t+1}\mfy_{t+1}+\eta_{t+1}^T\mfy_{t+1}\Big] \! + \!\!\! \mE_{\mfx_t|\mfN=N} \Big[ \frac{1}{2}\|\mfu_t\|^2 \Big] \\
& \; +\mE_{\mfy_t|\mfN=N}\Big[\frac{1}{2}\mfy_t^T Q \mfy_t+q^T\mfy_t -\frac{1}{2}\mfy_t^T P_t \mfy_t-\eta_t^T\mfy_t\Big]\\
& \; +\frac{1}{2}g_t^T\mfR_t^\dagger g_t-\frac{1}{2}\tr(P_{t+1}dd^T)-\eta_t^Td-\frac{1}{2}\tr(P_{t+1}\Sigma_w)\\
& \; -\frac{1}{2}\tr(P_{t+1}\Sigma_v)+\frac{1}{2}\tr(P_t\Sigma_v) - \frac{1}{2} \tr(Q \Sigma_v) \\
& =: \mE_{\mfy_{t:t+1}|\mfN=N} [ \tilde{\psi}_{t, N}(Q, q,  P_{t:t+1}, \eta_{t},  \xi_{t}; \mfy_{t:t+1}) ] \\
& \; + \mE_{\mfx_t|\mfN=N} [ \frac{1}{2}\|\mfu_t\|^2 ],
\end{align*}
where we introduce $\xi_t := g_t^T\mfR_t^\dagger g_t$. We construct the objective function $\Psi(Q, q,  P_{1:\nu}, \eta_{1:\nu},  \xi_{1:\nu-1})$ by summing up the above equation from $t=\nu-N+1$ to $\nu-1$, but excluding the terms $\mE_{\mfx_t|\mfN=N} [ \frac{1}{2}\|\mfu_t\|^2 ]$ which are constants, and taking the expectation over $\mfN$. In particular, 
\begin{align}
&\Psi(\cdot) := \sum_{N = 2}^{\nu}\mP(\mfN=N)\mE_{\mfy_{\nu-N+1:\nu}|\mfN=N}\left[\tilde{\psi}_N(\cdot) \right],\label{eq:stochastic_ioc_obj_rewriting}
\end{align}
where 
\begin{align}
& \tilde{\psi}_N(\cdot) = \sum_{t=\nu-N+1}^{\nu-1}\tilde{\psi}_{t,N}(\cdot) \nonumber \\
& = \frac{1}{2}\mfy_{\nu}^TP_{\nu}\mfy_{\nu} +\eta_{\nu}^T\mfy_{\nu} - \frac{1}{2}\mfy_{\nu-N+1}^TP_{\nu-N+1}\mfy_{\nu-N+1} \nonumber\\
&\; -\eta_{\nu-N+1}^T\mfy_{\nu-N+1} - \frac{1}{2}\tr(P_{\nu}\Sigma_v)+\frac{1}{2}\tr(P_{\nu-N+1}\Sigma_v)\nonumber\\
&\;+\sum_{t=\nu-N+1}^{\nu-1}\Big(\frac{1}{2}\xi_t -\frac{1}{2}\tr(P_{t+1}dd^T)-\eta_{t+1}^Td + \frac{1}{2} \mfy_t^TQ\mfy_t\nonumber \\
&\;  + q^T\mfy_t - \frac{1}{2}\tr(P_{t+1}\Sigma_w)-\frac{1}{2}\tr(Q\Sigma_v)\Big).
\end{align}
The objective function \eqref{eq:stochastic_ioc_obj_rewriting} can be rewritten as a joint expectation over $\mfy_{1:\nu}$ and $\mfN$. However, we find the form in \eqref{eq:stochastic_ioc_obj_rewriting} more useful both in analysis and in implementation.

The objective function \eqref{eq:stochastic_ioc_obj_rewriting} is constructed with the idea that $\Psi(\cdot) + \sum_{t = 1}^{\nu-1}  \mE_{\mfx_t, \mfN}\left[\frac{1}{2} \|\mfu_t\|^2\right]$, where the latter is the discarded constant, should be bounded from below by $0$ for all $(Q,q)\in\mscrF(I)$. 
Therefore, ideally we would consider the problem of finding the point  $(Q^\star,q^\star)$ that minimizes \eqref{eq:stochastic_ioc_obj_rewriting} subject to \eqref{eq:generalized_riccati_iterations} and \eqref{eq:psd_kernel_containment_cond}.
From Theorem~\ref{thm:indefinite_LQR}, we know that \eqref{eq:generalized_riccati_iterations} and \eqref{eq:psd_kernel_containment_cond} are equivalent to \eqref{eq:nes_suff_existence_solution} and $\xi_t = g_t^T\mfR_t^\dagger g_t$ for $t = 1:\nu-1$.
However, while \eqref{eq:stochastic_ioc_obj_rewriting} is linear and hence a convex function, neither the constraint \eqref{eq:nes_suff_existence_solution} nor the constraint $\xi_t = g_t^T\mfR_t^\dagger g_t$ for $t = 1:\nu-1$ are convex. Even so, note that the only nonconvex part in \eqref{eq:nes_suff_existence_solution} is \eqref{eq:nes_suff_existence_solution_3}. We therefore consider the relaxed convex problem obtained by removing the constraints \eqref{eq:nes_suff_existence_solution_3} and $\xi_t = g_t^T\mfR_t^\dagger g_t$ for $t = 1:\nu-1$.%
\footnote{Another possibility would be to substitute $\xi_t = g_t^T\mfR_t^\dagger g_t$ into the matrix \eqref{eq:nes_suff_existence_solution_2}, which would (also) give a convex problem; see \citep{nordstrom2011convexity, nordstrom2018note}.}
In particular, the optimization problem for IOC reads
\begin{subequations}\label{eq:stochastic_IOC_opt_pro}
\begin{align}
\min_{\substack{(Q, q) \in \mathbb{G}(\varphi)\\ \{P_{t} \in \mathbb{S}^n_+(\varphi) \}_{t = 1:\nu},\\ \{ \eta_{t} \in \mathscr{B}^{n}_{\varphi}(0) \}_{t = 1:\nu},\\ \{ \xi_t \in\mathscr{B}^{1}_{\varphi}(0) \}_{t=1:\nu-1}}} & \; \Psi(Q, q,  P_{1:\nu}, \eta_{1:\nu},  \xi_{1:\nu-1})\nonumber\\
\st \quad & \; P_{\nu} = Q, \label{eq:stochastic_IOC_opt_pro_first_const}\\
& \; \eta_{\nu}=q\label{eq:stochastic_IOC_opt_pro_second_const}\\
& 
H_t \succeq 0,\quad t=1:\nu-1,\label{eq:stochastic_IOC_opt_pro_last_const}
\end{align}
\end{subequations}
where $H_t$ has the same form as \eqref{eq:nes_suff_existence_solution_2}. In Section~\ref{subsec:consistency}, we prove that the unique optimal solution to this optimization problem is indeed $(\Qtrue, \qtrue)$.

Nevertheless, the distribution of $\bar{\mfx}, \{ \mfw_t \}, \{ \rvv_t \}$ and $\mfN$ are usually not a priori known in practice, and hence the distribution of $\mfy_t$ and $\mfN$ are not known. 
Therefore, it is not possible to calculate the objective function \eqref{eq:stochastic_ioc_obj_rewriting} explicitly, and hence we cannot solve \eqref{eq:stochastic_IOC_opt_pro} directly. But since we have the optimal state trajectory observations $\{y_t^i\}_{t=1}^{\nu}$ of the agent's trials, i.e., realizations of I.I.D.~random processes $\{\mfy_t^i\}_{t=1}^{\nu}$ for $i=1:M$, we can instead empirically estimate the objective function.
To this end, let $M_N$ denote the number of observations which has a planning horizon of $N$ time steps. Clearly, $\sum_{N = 2}^{\nu} M_N = M$. 
Then, for each value of $N$, the expectation in \eqref{eq:stochastic_ioc_obj_rewriting} is approximated by the empirical mean as
\begin{align*}
&\mE_{\mfy_{\nu-N+1:\nu}|\mfN=N}[\tilde{\psi}_{N}(\cdot)]\approx  \\
&\frac{1}{M_N}\!\! \sum_{i_N  = 1}^{M_N}\Big[
\frac{1}{2}\mfy_{\nu}^{i_N T} P_{\nu} \mfy_{\nu}^{i_N}\!\! +\!\!\eta_{\nu}^T\mfy_{\nu}^{i_N} \!\!- \!\! \frac{1}{2}\mfy_{\nu-N+1}^{i_NT} P_{\nu-N+1}\mfy_{\nu-N+1}^{i_N} \nonumber\\
&\; -\eta_{\nu-N+1}^T\mfy_{\nu-N+1}^{i_N} - \frac{1}{2}\tr(P_{\nu}\Sigma_v)+\frac{1}{2}\tr(P_{\nu-N+1}\Sigma_v)\nonumber\\
&\;+\sum_{t=\nu-N+1}^{\nu-1}\Big(\frac{1}{2}\xi_t -\frac{1}{2}\tr(P_{t+1}dd^T)-\eta_{t+1}^Td + \frac{1}{2} \mfy_t^{i_NT} Q\mfy_t^{i_N}\nonumber \\
&\;  + q^T\mfy_t^{i_N} - \frac{1}{2}\tr(P_{t+1}\Sigma_w)-\frac{1}{2}\tr(Q\Sigma_v)\Big)\Big]\Big].
\end{align*}
On the other hand, approximating the expectation over $\mfN$ is the same as estimating the probabilities $\mP(\mfN=N)$ using the empirical estimates  $M_N/M$.
This, together with the above expression, gives that 
\begin{align}
&\Psi(Q, q,  P_{1:\nu}, \eta_{1:\nu},  \xi_{1:\nu-1})\approx\nonumber\\
& \Psi_E^{\mfy}(Q, q,  P_{1:\nu}, \eta_{1:\nu},  \xi_{1:\nu-1})=\frac{1}{M} \!\! \sum_{N = 2}^{\nu} \sum_{i_N  = 1}^{M_N} \Big[ \frac{1}{2}\mfy_{\nu}^{i_NT}P_{\nu}\mfy_{\nu}^{i_N}\nonumber\\
& + \eta_{\nu}^T\mfy_{\nu}^{i_N} -\frac{1}{2}\mfy_{\nu-N+1}^{i_N T}P_{\nu-N+1}\mfy_{\nu-N+1}^{i_N} -\eta_{\nu-N+1}^T\mfy_{\nu-N+1}^{i_N}  \nonumber\\
&- \frac{1}{2}\tr(P_{\nu}\Sigma_v)+\frac{1}{2}\tr(P_{\nu-N+1}\Sigma_v)+\sum_{t=\nu-N+1}^{\nu-1} \Big( \frac{1}{2}\xi_t   \nonumber\\
&- \frac{1}{2}\tr(P_{t+1}dd^T)- \eta_{t+1}^Td+ \frac{1}{2} \mfy_t^{i_N T}Q\mfy_t^{i_N}+ q^T\mfy_t^{i_N}  \nonumber\\
&  - \frac{1}{2}\tr(P_{t+1}\Sigma_w) -\frac{1}{2}\tr(Q\Sigma_v)\Big) \Big]. \label{eq:obj_fun_estimator}
\end{align}
We therefore consider the estimator
\begin{align}
\min_{\substack{(Q, q) \in \mathbb{G}(\varphi)\\ \{P_{t} \in \mathbb{S}^n_+(\varphi) \}_{t = 1:\nu},\\ \{ \eta_{t} \in \mathscr{B}^{n}_{\varphi}(0) \}_{t = 1:\nu},\\ \{ \xi_t \in\mathscr{B}^{1}_{\varphi}(0) \}_{t=1:\nu-1}}} & \; \Psi_E^{\mfy}(Q, q,  P_{1:\nu}, \eta_{1:\nu},  \xi_{1:\nu-1})\nonumber\\
\st & \;  \text{\eqref{eq:stochastic_IOC_opt_pro_first_const}--\eqref{eq:stochastic_IOC_opt_pro_last_const} hold.}\label{eq:stochastic_ioc_approximation}
\end{align}
In practice, an estimate is obtained by solving \eqref{eq:stochastic_ioc_approximation} for a given realization $\{ y_{1:\nu}^i \}_{i = 1}^M$ of $\{ \mfy_{1:\nu}^i \}_{i = 1}^M$. We will use the notation $\Psi_E^{\mfy}(\cdot)|_{\mfy = y}$ to denote the objective function at the given realization.

\subsection{Statistical consistency analysis}\label{subsec:consistency}
In this section,  we analyze the statistical consistency of the IOC algorithm.
To proceed, we first show that the optimization problem \eqref{eq:stochastic_IOC_opt_pro} is well-posed, i.e., the objective function \eqref{eq:stochastic_ioc_obj_rewriting} is bounded from below on its feasible domain \eqref{eq:stochastic_IOC_opt_pro_first_const}, \eqref{eq:stochastic_IOC_opt_pro_second_const} and \eqref{eq:stochastic_IOC_opt_pro_last_const}. 
In addition, we show the ``true" $(\Qtrue, \qtrue)$ is actually the unique global minimizer.

\begin{theorem}\label{thm:IOC_Q_optimal}
Let $(\bar{Q},\bar{q})\in\mscrF(I)$ be the ``true" parameters of the stochastic linear-quadratic control problem \eqref{eq:stochastic_forward_problem} that governs the agent, and let $\mfx_{1:\nu}$, $\mfu_{1:\nu}$, and $\mfy_{1:\nu}$ be distributed accordingly.
Under Assumptions~\ref{ass:controlability_and_full_rank}, \ref{ass:IID}, \ref{ass:planning_horizon}, \ref{ass:persistent_excitation}, and \ref{ass:bounded_parameter}, for any feasible solution $(Q,q,P_{1:\nu},\eta_{1:\nu},\xi_{1:\nu-1})$ of the optimization problem \eqref{eq:stochastic_IOC_opt_pro},  the objective function \eqref{eq:stochastic_ioc_obj_rewriting} is bounded from below by $- \sum_{t = 1}^{\nu-1}  \mE_{\mfx_t, \mfN}\left[\frac{1}{2} \|\mfu_t\|^2\right]$.
Moreover, for $\varphi$ that is large enough, $(\bar{Q}, \bar{q},\Ptrue_{1:\nu},\etatrue_{1:\nu},\xitrue_{1:\nu-1})$ is the unique globally optimal solution achieving the lower bound, where $\Ptrue_{1:\nu},\etatrue_{1:\nu}$ are generated by \eqref{eq:generalized_riccati_iterations} and $\xitrue_t=\gtrue_t^T\mfRtrue_t^\dagger\gtrue_t$, $t=1:\nu-1$.
\end{theorem}
\begin{proof}
See Appendix.
\end{proof}

Having shown that the optimization problem \eqref{eq:stochastic_IOC_opt_pro} has $(\Qtrue, \qtrue)$ as unique globally optimal solution, next we turn to the estimator  \eqref{eq:stochastic_ioc_approximation}. We show that it is statistically consistent, but to this end we first have the following Lemmas.

\begin{lemma}[Boundedness of estimator]\label{lem:bounded_domain_and_cost}
The feasible region in problem \eqref{eq:stochastic_ioc_approximation} is compact. Moreover, for any realization, the cost function $\Psi_E^{\mfy}(Q, q,  P_{1:\nu}, \eta_{1:\nu},  \xi_{1:\nu-1})|_{\mfy=y}$ is bounded on the feasible region, and the optimization problem \eqref{eq:stochastic_ioc_approximation} is convex and admits an optimal solution.
\end{lemma}

\begin{proof}
It has already been established that the feasible region is convex, see the arguments in Section~\ref{subsec:construction_and_approximation}.
Moreover, the feasible region is bounded by construction. Next, for any realization,  $\Psi_E^{\mfy}(Q, q,  P_{1:\nu}, \eta_{1:\nu},  \xi_{1:\nu-1})|_{\mfy=y}$ is a linear function, and hence \eqref{eq:stochastic_ioc_approximation} is a convex problem. Moreover, since the cost function is linear, it is bounded on the compact feasible domain. Finally, by Weierstrass' theorem  (see, e.g., \citep[Prop.~A.8]{bertsekas1999nonlinear}) the optimization problem \eqref{eq:stochastic_ioc_approximation} admits an optimal solution.
\end{proof}

\begin{lemma}[Uniform law of large numbers]\label{lem:ulln}
For large enough $\varphi$ and under Assumptions \ref{ass:controlability_and_full_rank}, \ref{ass:IID}, \ref{ass:planning_horizon}, \ref{ass:persistent_excitation}, and \ref{ass:bounded_parameter}, the optimal value
\begin{align*}
\sup_{\substack{(Q, q) \in \mathbb{G}(\varphi)\\ \{P_{t} \in \mathbb{S}^n_+(\varphi) \}_{t = 1:\nu},\\ \{ \eta_{t} \in \mathscr{B}^{n}_{\varphi}(0) \}_{t = 1:\nu},\\ \{ \xi_t \in\mathscr{B}^{1}_{\varphi}(0) \}_{t=1:\nu-1}}}
& \; | \Psi_E^{\mfy}(\cdot) - \Psi(\cdot) | \\
\st \quad
& \;  \text{\eqref{eq:stochastic_IOC_opt_pro_first_const}--\eqref{eq:stochastic_IOC_opt_pro_last_const} hold.}
\end{align*}
converges to $0$ almost surely as $M \to \infty$.
\end{lemma}

\begin{proof}
The proof follows along the lines of \citep[Proof of Lem.~4.2]{zhang2021inverse}. First, the argument inside the expectation in $\Psi(\cdot)$ can be bounded from above by an integrable function of the random variables, but which is independent of the parameters $(Q, q,  P_{1:\nu}, \eta_{1:\nu},  \xi_{1:\nu-1})$. This can be done using bounds from Lemmas~\ref{lem:stochastic_persistent_excitation} and \ref{lem:bounded_domain_and_cost}. Using this bound in terms of an integrable function, the result follows from \citep[Thm.~2]{jennrich1969asymptotic}.
\end{proof}

\begin{theorem}[Statistical consistency]\label{thm:statistical_consistency}
For large enough $\varphi$ and under Assumptions \ref{ass:controlability_and_full_rank}, \ref{ass:IID}, \ref{ass:planning_horizon}, \ref{ass:persistent_excitation}, and \ref{ass:bounded_parameter}, given a realization of $M$ trajectories, let $(Q^{M},q^{M}, P_{1:\nu}^{M}, \eta_{1:\nu}^{M},  \xi_{1:\nu-1}^{M})$ be a corresponding optimal solution to \eqref{eq:stochastic_ioc_approximation}. Then $Q^{M} \overset{p}\rightarrow \Qtrue$ and $q^{M} \overset{p}\rightarrow \qtrue$ as $M\rightarrow \infty$.
\end{theorem}

\begin{proof}
The result follows by verifying the conditions in \citep[Thm.~5.7]{van1998asymptotic}. In particular, since \eqref{eq:stochastic_ioc_approximation} is convex,  $(Q^{M},q_{1:\nu}^{M}, P_{1:\nu}^{M}, \eta_{1:\nu}^{M},  \xi_{1:\nu-1}^{M})$ is a globally optimal solution. This means that
\begin{align*}
& \Psi_E^{\mfy}(Q^{M},q^{M}, P_{1:\nu}^{M}, \eta_{1:\nu}^{M},  \xi_{1:\nu-1}^{M})|_{\mfy=y} \\
& \leq \Psi_E^{\mfy}(\Qtrue, \qtrue, \Ptrue_{1:\nu}, \etatrue_{1:\nu},  \xitrue_{1:\nu-1})|_{\mfy=y}.
\end{align*}
Moreover, since convergence almost surely implies convergence in probability \citep[Lem.~3.2]{kallenberg1997foundations}, Lemma~\ref{lem:ulln} implies that 
\begin{align*}
\sup_{\substack{(Q, q) \in \mathbb{G}(\varphi)\\ \{P_{t} \in \mathbb{S}^n_+(\varphi) \}_{t = 1:\nu},\\ \{ \eta_{t} \in \mathscr{B}^{n}_{\varphi}(0) \}_{t = 1:\nu},\\ \{ \xi_t \in\mathscr{B}^{1}_{\varphi}(0) \}_{t=1:\nu-1}}}
& \; | \Psi_E^{\mfy}(\cdot) - \Psi(\cdot) | \\
\st \quad
& \;  \text{\eqref{eq:stochastic_IOC_opt_pro_first_const}--\eqref{eq:stochastic_IOC_opt_pro_last_const} hold.}
\end{align*}
converges to $0$ in probability as $M \to \infty$.
Finally, the fact that the feasible region to \eqref{eq:stochastic_IOC_opt_pro} and \eqref{eq:stochastic_ioc_approximation} is compact (see Lemma~\ref{lem:bounded_domain_and_cost}), and that \eqref{eq:stochastic_IOC_opt_pro} has a unique optimal solution (see Theorem~\ref{thm:IOC_Q_optimal}),  by \citep[p.~46]{van1998asymptotic} the last condition also holds. Hence, the result follows.
\end{proof}

\subsection{On implementation and the computational complexity of the estimator}
To get a point estimate from the estimator \eqref{eq:stochastic_ioc_approximation}, the data (i.e., the observed trajectories) are used in the optimization problem \eqref{eq:stochastic_ioc_approximation}. This problem can be solved using any appropriate method for solving the convex optimization problem. In fact, even for modest sized problems, one can use many off-the-shelf convex optimization solvers. For example, problem  \eqref{eq:stochastic_ioc_approximation} can be implemented, almost exactly as stated, in frameworks for disciplined convex programming, including YALMIP \citep{Lofberg2004}, CVX \citep{gb08, cvx}, and CVXPY \citep{diamond2016cvxpy, agrawal2018rewriting}.

The only difference between an implementation in a framework for disciplined convex programming and the stated problem in \eqref{eq:stochastic_ioc_approximation}, is that the cost function can be rewritten to make the implementation more efficient. To this end, observe that for any $Z \in \mathbb{S}^n$ and any $a \in \mR^n$, $a^T Z a = \tr (Z aa^T)$. This means that the objective function \eqref{eq:obj_fun_estimator} can be rewritten as
\begin{align*}
& \Psi_E^{\mfy}(Q, q,  P_{1:\nu}, \eta_{1:\nu},  \xi_{1:\nu-1})= \frac{1}{M}  \!\! \sum_{N = 2}^{\nu} \Big[  \frac{1}{2} \tr ( P_{\nu} \mfY_\nu^{(N)}) + \eta_{\nu}^T \mfy_\nu^{(N)} \nonumber\\
&  -\frac{1}{2}  \tr ( P_{\nu-N+1} \mfY_{\nu-N+1}^{(N)} ) -\eta_{\nu-N+1}^T\mfy_{\nu-N+1}^{(N)} - \frac{M_N}{2} \tr(P_{\nu}\Sigma_v)  \nonumber\\
& +  \frac{M_N}{2}\tr(P_{\nu-N+1}\Sigma_v)+  \sum_{t=\nu-N+1}^{\nu-1} \Big( \frac{M_N}{2}\xi_t - \frac{M_N}{2}\tr(P_{t+1}dd^T)  \nonumber\\
&- M_N \eta_{t+1}^Td+ \frac{1}{2}  \tr ( Q\mfY_t^{(N)} )+ q^T\mfy_t^{(N)}  \nonumber\\
&  - \frac{M_N}{2}\tr(P_{t+1}\Sigma_w) -\frac{M_N}{2}\tr(Q\Sigma_v)\Big) \Big],
\end{align*}
where  $\mfy_t^{(N)} =  \sum_{i_N  = 1}^{M_N} \mfy_{t}^{i_N}$ and $\mfY_t^{(N)} =  \sum_{i_N  = 1}^{M_N} \mfy_{t}^{i_N}  (\mfy_{t}^{i_N})^T$.
Note that $\mfy_t^{(N)}$ and $\mfY_t^{(N)}$ are collecting all the samples at time-point $t$ from trajectories with the same planning horizon length $N$, and that these can be pre-computed from the data before assembling the optimization problem \eqref{eq:stochastic_ioc_approximation}. Moreover, the sizes of $\mfy_t^{(N)}$ and $\mfY_t^{(N)}$ only depend on the dimension of the state space, $n$. It means that the size of the optimization problem does not grow with the amount of data collected.

More specifically, since $Q \in \mathbb{S}^n$, $q \in \mR^n$, $\{P_{t} \in \mathbb{S}^n_+\}_{t = 1:\nu}$, $\{ \eta_{t} \in \mR^{n} \}_{t = 1:\nu}$, and $ \{ \xi_t \in \mR \}_{t=1:\nu-1}$, the number of variables in the problem is $n(n+1)/2 + n + \nu n(n+1)/2 + \nu n + \nu$.
Moreover, the LMI constraints in \eqref{eq:stochastic_IOC_opt_pro_last_const} are $\nu$ symmetric matrices of size $(m + n + 1) \times (m + n + 1)$. This means that, e.g., $n = 12$, $m= 4$, and $\nu = 80$ gives a problem with a total of $7370$ scalar variables and $80$ LMI constraints of size $17 \times 17$. As we demonstrate in Section~\ref{subsec:example_large_system}, this can be handled by standard off-the-shelf convex optimization solvers.

\section{Numerical examples}\label{sec:numerical_examples}
In this section, we present two numerical examples. The first example, in Section~\ref{subsec:example_large_system}, illustrates that the problem \eqref{eq:stochastic_ioc_approximation} can be solved efficiently with off-the-shelf convex optimization solvers. The second example, in Section~\ref{subsec:pursuit-evasion}, applies the developed methodology to a non-zero sum pursuit-evasion game, where the pursuer models the evaders objective function using collected data. In both examples, the problem is solved on a MacBook Pro with Apple M1 eight-core CPU and 16GB of RAM,  and the implementation is done using YALMIP \citep{Lofberg2004} in Matlab and solved by MOSEK \citep{mosek}.

\subsection{Demonstration of performance for a system with both modest size and planning horizon}\label{subsec:example_large_system}

To illustrate the performance of the method, we generate a system with modest size and modest planning horizon length. In particular, to ensure that Assumption~\ref{ass:controlability_and_full_rank} holds, we generate continuous-time matrices $\hat{A}\in\mR^{12\times 12}$ and $\hat{B}\in\mR^{12 \times 4}$ in controllable canonical form
\begin{align*}
\hat{A} = \begin{bmatrix}
 & I_4 \\
& & I_4 \\
 & & & I_4\\
a_1  I_4  &a_2  I_4 & a_3  I_4 & a_{4} I_4
\end{bmatrix},\quad 
\hat{B} = 
\begin{bmatrix}
\mathbf{0}_4 \\ \mathbf{0}_4 \\ \mathbf{0}_4 \\I_4
\end{bmatrix}.
\end{align*}
We sample the coefficients $a_i$, $i=1:4$ from a standard normal distribution $\mathcal{N}(0,1)$. Next, we discretize the system by letting $A = e^{\hat{A}\Delta t}$ and $B =\int_0^{\Delta t} e^{\hat{A}t}dt\hat{B}$, using the sampling period $\Delta t = 0.1$. We choose $\Qtrue$ to be the Hermitian part of a randomly drawn matrix with shifted eigenvalues so that the smallest eigenvalue is $-0.1$.%
\footnote{We shift the eigenvalues in order to make sure that we get a forward problem that is well-posed, i.e., that $(\Qtrue,\qtrue) \in \mscrF(I)$.}
Namely, we let $G^\prime = (G + G^T)/2$ and  $\Qtrue = G^\prime - (\sigma_{\min}(G^\prime) + 0.1)I$, where $\sigma_{\min}(\cdot)$ is the smallest eigenvalue of a matrix and where $G \in\mR^{12\times 12}$ and elements are randomly drawn from $\mathcal{N}(0,1)$. We set $\nu=80$, and verify that the conditions in point 2) in Theorem~\ref{thm:indefinite_LQR} hold, i.e., that $\mfRtrue_t \succ 0$ for $t = 1 : 79$. The process noise $\mfw_t$ and measurement noise $\rvv_t$ are drawn from multi-variate normal distribution $\mathcal{N}(0,\Sigma_w)$ and $\mathcal{N}(0,\Sigma_v)$, respectively, with covariance matrices that are randomly generated from a Wishart distribution of degree $12$, i.e., with the same degrees of freedom as the dimension of the state.
Moreover, the Wishart distribution used to draw the covariance matrices has itself a random covariance of $0.01GG^T$, where each element in $G\in \mR^{12 \times12}$ was drawn from a standard normal distribution.
 Finally, we generate $M=5\times 10^4$ optimal trajectories, with the planning horizon lengths $N$ drawn uniformly from the integers in the interval $[2, 80]$ and with initial value $\bar{x}$ drawn from $\mathcal{N}(0, 100 I_{12})$.

The time to solve the optimization problem \eqref{eq:stochastic_ioc_approximation}, as reported by MOSEK, is $4.85$ seconds. Moreover, the relative error of the estimate is $0.0347$, where the relative error is computed as $\frac{\|\tilde{Q}_{est}-\tilde{\Qtrue}\|_F}{\| \tilde{\Qtrue} \|_F}$, where $\tilde{\Qtrue}$ is defined in \eqref{eq:extended_cost_mtx} and $\tilde{Q}_{est}$ is defined analogously. This shows that solving the IOC problem for systems of  ``moderate" size and planning horizon length can be done efficiently with off-the-shelf solvers.

\subsection{Identification of cost in non-zero sum pursuit-evasion game}\label{subsec:pursuit-evasion}

In this section, we demonstrate the performance of the proposed IOC algorithm on a non-zero sum two-dimensional finite-horizon linear-quadratic pursuit-evasion game, cf.~\citep{starr1969nonzero}. For a more extensive treatment of pursuit-evasion games, see, e.g., \citep{bacsar1982dynamic}. To this end, let $\mfx_t\in\mR^2$ be the distance between the pursuer and the evader, and let $\mfu_t^p,\mfu_t^e\in\mR^2$ be the control signal of the pursuer and the evader, respectively. 
In particular, for each realization $(\bar{x},N)$ of $(\bar{\mfx},\mfN)$, we assume that the evader solves the following problem 
\begin{subequations}\label{eq:stochastic_forward_problem_example}
\begin{align}
\min_{\substack{\mfx_{1:\nu} , \\ \mfu_{1:\nu}^e}}
  \;
  & \; J_N := \mE_{\mfw_{\nu-N+1:\nu-1}} \Big[\frac{1}{2} \mfx_{\nu}^{T} Q^e \mfx_{\nu} \!+\!\! \sum_{t = \nu - N+1}^{\nu-1} [\frac{1}{2} \mfx_{t}^{T} Q^e \mfx_{t}\nonumber\\
  & \;\;  + \frac{1}{2}\|\mfu_t^e\|^2 ] \Big]\label{eq:evader_cost}\\
  \st
  & \; \mfx_{t+1} = A\mfx_{t} + B\mfu_{t}^e+B\mfu_t^p+\mfw_t, \nonumber\\
  &\quad t = \nu\!-\!N\!+\!1:\nu\!-\!1, \label{eq:evader_dynamics} \\
  & \; \mfx_{t + 1} = \mfx_t, \quad t=1:\nu-N\label{eq:evader_before_init_cond} \\
  & \; \mfx_1 = \bar{x}, \label{eq:evader_init_cond}\\
  & \; \mfu_1^e=\ldots=\mfu_{\nu-N}^e = 0,\label{eq:evader_stand_still}
\end{align}
\end{subequations}
where $(A,B)$ is discretized in the same way as in Sec.~\ref{subsec:example_large_system} from the continuous-time dynamics $\dot{\mfx}=\hat{A}\mfx+\hat{B}\mfu^e+\hat{B}\mfu^p$ using the sampling period $\Delta t = 0.1$, and where $\nu=20$.
Notably,  $Q^e \prec 0$.
In practice, as a pursuer, $Q^e$ is unknown. In order to gain advantages over the evader and predict its future movements, as a pursuer, one can first use some ``trivial" and ``dummy" movements $\mfu_{\nu-N+1:\nu}^p$ that are easy for the evader to predict (i.e., known by the evader) in the first a few rounds of the game. During these rounds, the pursuer collects the evader's behaviour data and use the proposed IOC algorithm to estimate $Q^e$. In particular, here we assume the pursuer choose control $\mfu_t^p$ to be a constant during the data collection phase for convenience. Consequently, the forcing term $d = B\mfu_t^p$ would be constant for the evader (cf.~\eqref{eq:stochastic_forward_problem_dynamics}).%
\footnote{In fact, it does not matter what kinds of strategy the pursuer uses in the data collection phase, as long as the evader can foresee it, since, 
as mentioned in Sec.~\ref{rem:forward_problem_formulation_rem}, the results still hold for time-varying $d_t$.}
The pursuer observes the noisy distance (see \eqref{eq:noisy_observation}) between the pursuer and evader, which is the optimal solution to \eqref{eq:stochastic_forward_problem_example}.

To simulate this, we choose $\hat{A} = 0$, $\hat{B} = I_2$, $Q^e=-0.1I_2$, and for each time step in the trajectories process noise $\mfw_t$ and measurement noise $\rvv_t$ are drawn  from multi-variate normal distribution $\mathcal{N}(0,\Sigma_w)$ and $\mathcal{N}(0,\Sigma_v)$, respectively, with covariance matrices
\[
\Sigma_w \approx 10^{-2}
\begin{bmatrix}
1.04 & 0.68 \\
0.68 & 1.00
\end{bmatrix}
\; \text{and} \quad
\Sigma_v \approx 10^{-2}
\begin{bmatrix}
2.33 &  -2.25 \\
-2.25 & 2.18
\end{bmatrix}.
\]
The latter matrices were randomly generated by drawing two elements from a Wishart distribution of degree 2, i.e., with the same degrees of freedom as the dimension of the state. 
The Wishart distribution is itself generated analogously to the distribution in Section~\ref{subsec:example_large_system}.
As ``dummy" movements for the pursuer, we choose $\mfu_t^p = [-1,-1]^T,\;t=\nu-N+1:\nu$, and hence the constant forcing term in the dynamics is given by $d=B[-1,-1]^T$. Finally, the random variable $\mfN$ is taken to be uniformly distributed on the integers between $2$ and $\nu = 20$.

To test the performance of the algorithm, we generate $100$ batches of trajectories, where each batch consists of $50 000$ trajectories. For each batch, we divide the trajectories into groups of size $M = 100 + 100 (k-1) $, for $k = 1, \ldots, 500$,  where each larger group contains all the trajectories of a smaller group. For each such group of trajectories, we solve the IOC problem (with $\varphi$ set to $10^{6}$), and this procedure is repeated for all the $100$ batches. This means that we obtain estimates $Q_{est}^{\ell, M}$, for $\ell = 1, \ldots, 100$ and $M = 100, 200, \ldots, 50000$. For each value of $M$, the relative error $\|Q_{est}^{\ell, M} - Q^e\|_F / \| Q^e\|_F$ is averaged over the batches, and the resulting empirical mean and empirical standard deviation (as a function of $M$) is shown in Figure~\ref{fig:convergence_rel_error_mean_std}.
From the figure we see that, in line with the statistical consistency proved in Theorem~\ref{thm:statistical_consistency}, both the mean and the standard deviation decreases with increasing $M$.
Moreover, in Figure~\ref{fig:convergence_rel_error_mean_std} the logarithm of the mean and the logarithm of the standard deviation appears to be (approximately) affine in $\log(M)$. The figure also shows the corresponding lines obtained by fitting an affine model to each of the two sets of logarithmic data. From this fit, we see that $\texttt{Mean of relative error} \approx \mathcal{O}(M^{-0.48})$ and $\texttt{Standard deviation of relative error} \approx \mathcal{O}(M^{-0.55})$. We hence suspect that the convergence rate is $\mathcal{O}(M^{-0.5})$, and that $\sqrt{M}(Q_M-\Qtrue)$ is asymptotically normal, just like most M-estimators such as maximum log-likelihood \citep[p.~51]{van1998asymptotic}. However, a theoretical analysis of this is left for future work.

\begin{figure}[!htpb]
    \centering
    \includegraphics[trim=1.5cm 0cm 2.cm 0.2cm, clip, width = \columnwidth]{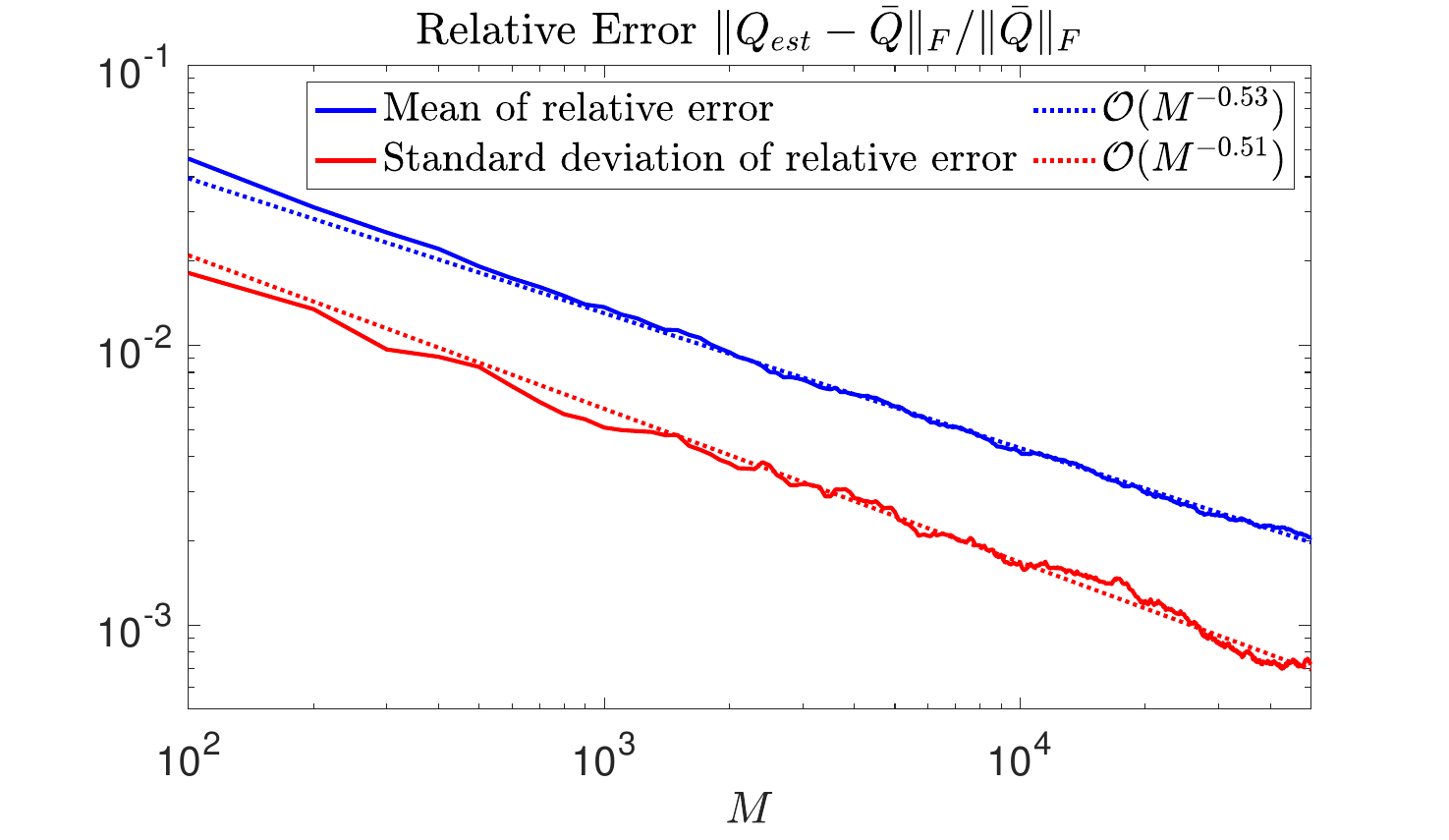}
    \caption{Log-log plot of the mean and standard deviation of the relative error of $Q_{est}$ as a function of the number of trajectories. The estimates are obtained using noisy data, as described in Section~\ref{subsec:pursuit-evasion}.}
    \label{fig:convergence_rel_error_mean_std}
\end{figure}

\section{Conclusion}\label{sec:conclusion}
In this work, we have considered the inverse optimal control problem for discrete-time finite-horizon general indefinite linear-quadratic problems with stochastic planning horizons. We first investigate the necessary and sufficient conditions for when the forward problem is solvable. The identifiability of the corresponding inverse optimal control problem is analyzed and proved. Furthermore, based on the underlying necessary and sufficient condition, we construct the estimator of the inverse optimal control problem as the solution to a convex optimization problem, and prove that the estimator is statistically consistent. The performance of the estimator is illustrated on a numerical example of identifying the evaders cost in non-zero sum pursuit-evasion game.

\begin{ack}                               
The authors would like to thank the anonymous reviewers for helpful comments that helped improve the paper.
\end{ack}

\appendix

\section{Deferred proofs}

\begin{proof}[Proof of Proposition~\ref{prop:stability_longest_horizon}.]
We show the contraposition of the statement, i.e., that if  $(\Qtrue, \qtrue,\Rtrue)$ is such that there exists an initial value $\bar{x}\in\mR^n$ and a time-horizon length $N\in\{2,\ldots,\nu\}$ so that the optimal control problem \eqref{eq:stochastic_forward_problem} is unbounded from below, then there exists an initial value $\bar{x}^\prime$ so that \eqref{eq:stochastic_forward_problem} is unbounded from below for planning horizon length $\nu$. To this end, consider planning horizon $\mfN = \nu$. By Theorem~\ref{thm:indefinite_LQR}, if $(\Qtrue, \qtrue) \not \in \mscrF(\Rtrue)$ then there exists an $N$ such that 
\eqref{eq:psd_iff_cond} or \eqref{eq:kernel_containment} does not hold for $t = \nu - N +1$. Splitting the summation in the objective function as
\begin{align*}
J_\nu & = \mE_{\mfw_{1:\nu-1}} \Big[\frac{1}{2} \mfx_{\nu}^{T} \Qtrue \mfx_{\nu} +\qtrue^T\mfx_{\nu}  \nonumber\\
& +\sum_{t = \nu - N+1}^{\nu-1} [\frac{1}{2} \mfx_{t}^{T} \Qtrue \mfx_{t} +\qtrue^T\mfx_t + \frac{1}{2}\mfu_{t}^T\Rtrue\mfu_t ]  \\
& + \sum_{t = 1}^{\nu - N} [\frac{1}{2} \mfx_{t}^{T} \Qtrue \mfx_{t} +\qtrue^T\mfx_t + \frac{1}{2}\mfu_{t}^T\Rtrue\mfu_t ] \Big], \\
\end{align*}
and following along the lines of the proof of ``1)$\implies$2)'' in Theorem~\ref{thm:indefinite_LQR}, similar to  \eqref{eq:JN_inequality} we get the following inequality
\begin{align*}
J_\nu & \le \mE_{\mfw_{1:\nu-1}} \Big[\frac{1}{2}\mfu_{\nu-N+1}^T\mfRtrue_{\nu-N+1}\mfu_{\nu-N+1}\\
&+\mfx_{\nu-N+1}^T\mfStrue_{\nu-N+1}^T\mfu_{\nu-N+1}+\gtrue_{\nu-N+1}^T\mfu_{\nu-N+1}\\
&+\mfx_{\nu-N+1}^T\mfStrue_{\nu-N+1}^T\mfRtrue_{\nu-N+1}^\dagger\mfStrue_{\nu-N+1}\mfx_{\nu-N+1}\\
&+\gtrue_{\nu-N+1}^T\mfRtrue_{\nu-N+1}^\dagger \mfStrue_{\nu-N+1}\mfx_{\nu-N+1}+\tau(\mfx_{\nu-N+1})\\
&+ \sum_{t = 1}^{\nu - N} [\frac{1}{2} \mfx_{t}^{T} \Qtrue \mfx_{t} +\qtrue^T\mfx_t + \frac{1}{2}\mfu_{t}^T\Rtrue\mfu_t ]\Big].
\end{align*}
Moreover, from the same proof we know that we can select $\mfx_{\nu-N+1}$ in order to make the terms outside of the last summation unbounded from below. Now, by invertability of $A$, we can select $\mfu_t = 0$ and $\mfx_t = A^{-1}(\mfx_{t+1} - d - \mfw_t)$  for $t = 1:\nu-N$. For any value of $\mfx_{\nu-N+1}$, this gives an initial condition and a sequence of states and controls that fulfill the constraints \eqref{eq:stochastic_forward_problem_dynamics}-\eqref{eq:stochastic_control_stand_still} (note that \eqref{eq:stochastic_forward_problem_before_init_cond} and \eqref{eq:stochastic_control_stand_still} are vacuous since $\mfN = \nu$). Moreover it is easily seen that $J_\nu$ is bounded from above by an expression similar to the one in the previous proof, but containing an additional constant $\tilde{\tau}(\mfx_{\nu-N+1})$. Following a logic similar to the reminder of the proof of ``1)$\implies$2)'' in Theorem~\ref{thm:indefinite_LQR}, the result follows.
\end{proof}

\begin{proof}[Proof of Lemma~\ref{lem:positive_definite_cov}.]
Let $N \in \{2, \ldots, \nu\}$ be such that $\mP(\mfN = N) > 0$. The matrix $\cov_{\bar{\mfx} \mid \mfN = N}(\bar{\mfx},\bar{\mfx})$ is symmetric. Let $\cov_{\bar{\mfx} \mid \mfN = N}(\bar{\mfx},\bar{\mfx}) = \sum_{i = 1}^n \lambda_i v_i$ be an orthonormal eigen-decomposition of the matrix. This means that we can write $\bar{\mfx} = \sum_{i=1}^n \bm{\alpha}_i v_i$ for some real-valued random variables $\bm{\alpha}_i$,  $i = 1, \ldots, n$.

Assume that $\cov_{\bar{\mfx} \mid \mfN = N}(\bar{\mfx},\bar{\mfx})$ is not (strictly) positive definite. Then at least one eigenvalue is zero; without loss of generality, let $\lambda_1 = 0$. Then we have that
\begin{align*}
0 & = v_1^T \cov_{\bar{\mfx} \mid \mfN = N}(\bar{\mfx},\bar{\mfx}) v_1 \\
& = \mE_{\bar{\mfx} \mid \mfN = N}(v_1^T \bar{\mfx}\bar{\mfx}^T v_1) -  \mE_{\bar{\mfx} \mid \mfN = N}(v_1^T \bar{\mfx})  \mE_{\bar{\mfx} \mid \mfN = N}(\bar{\mfx}^T v_1) \\
&= \mE_{\bar{\mfx} \mid \mfN = N}(\bm{\alpha}_1^2) -  \mE_{\bar{\mfx} \mid \mfN = N}(\bm{\alpha}_1)  \mE_{\bar{\mfx}|\mfN=N}(\bm{\alpha}_1),
\end{align*}
and thus that $(\mE_{\bar{\mfx} \mid \mfN = N}(\bm{\alpha}_1))^2 = \mE_{\bar{\mfx}|\mfN=N}(\bm{\alpha}_1^2)$. By Jensen's inequality \citep[Prop.~9.24]{bauschke2017convex}, we know that $(\mE_{\bar{\mfx} \mid \mfN = N}(\bm{\alpha}_1))^2 \leq \mE_{\bar{\mfx} \mid \mfN = N}(\bm{\alpha}_1^2)$, and by following the proof of \citep[Prop.~9.24]{bauschke2017convex}, we see that equality holds if and only if $\bm{\alpha}_1$ is constant a.s.~(otherwise, the inequality from \citep[Thm.~9.23]{bauschke2017convex} is strict at some points). To this end, let $\bm{\alpha}_1 = c$ a.s.~for some constant $c$. If $c = 0$, then for $\chi = v_1$, Assumption~\ref{ass:persistent_excitation} does not hold. If $c \neq 0$, then the probability mass of $\bar{\mfx}$ is located on a hyperplane defined by $\bm{\alpha}_1 = c$, which does not pass through the origin. In this case, let $\chi = v_2$ and note that for $\epsilon < c$ we have that $\mP(\bar{\mfx} \in \mathscr{B}^{n}_{\epsilon}(\rho v_2) ) = 0$ for all $\rho$, hence violating Assumption~\ref{ass:persistent_excitation}.
Therefore, we must have $\cov_{\bar{\mfx} \mid \mfN = N}(\bar{\mfx},\bar{\mfx}) \succ 0$.
\end{proof}

\begin{proof}[Proof of Lemma~\ref{lem:stochastic_persistent_excitation}.]
Note that
\begin{align*}
&\cov_{\bar{\mfx}|\mfN=N}(\bar{\mfx},\bar{\mfx}) = \mE_{\bar{\mfx}|\mfN=N}[\bar{\mfx}\bar{\mfx}^T]-\mE_{\bar{\mfx}|\mfN=N}[\bar{\mfx}]\mE_{\bar{\mfx}|\mfN=N}[\bar{\mfx}]^T\\
&=\underbrace{\mE_{\bar{\mfx}|\mfN=N}\left[
\begin{bmatrix}
\bar{\mfx}\bar{\mfx}^T &\bar{\mfx}\\
\bar{\mfx}^T &1
\end{bmatrix}
\right]}_{\mE_{\bar{\mfx}|\mfN=N}[\bar{\tilde{\mfx}}\bar{\tilde{\mfx}}^T]}\backslash 1.
\end{align*}
Hence by Lemma \ref{lem:positive_definite_cov} and  \citep[Thm.~1.12]{horn2005basic},  we know that for all $N$ such that $\mP(\mfN = N) > 0$, $\mE_{\bar{\mfx}|\mfN=N}[\tilde{\bar{\mfx}}\tilde{\bar{\mfx}}^T]\succ 0$.
On the other hand, by Assumption \ref{ass:IID}, $\mfw_t$ is uncorrelated with
the noiseless $\mfz_t := A\mfx_t+B\mfu_t + d$, for $t = \nu-N+1:\nu-1$, and for such $t$ it thus hold that
\begin{align*}
&\mE_{\mfx_{t+1}|\mfN=N}\left[\tilde{\mfx}_{t+1}\tilde{\mfx}_{t+1}^T\right]\nonumber\\
&=\mE_{\mfx_{t+1}|\mfN=N}\left[
\begin{bmatrix}
\mfz_t+\mfw_t\\1
\end{bmatrix}
\begin{bmatrix}
\mfz_t^T+\mfw_t^T &1
\end{bmatrix}
\right]\nonumber\\
&=\mE_{\mfx_t|\mfN=N}[\tilde{\mfz}_{t} \tilde{\mfz}_{t}^T]+
\begin{bmatrix}
\Sigma_w &0\\
0 &0
\end{bmatrix} \\
& = \tilde{A}_{cl}(t;\Qtrue,\qtrue) \mE_{\mfx_t|\mfN=N}[\tilde{\mfx}_{t} \tilde{\mfx}_{t}^T] \tilde{A}_{cl}(t;\Qtrue,\qtrue)^T+
\begin{bmatrix}
\Sigma_w &0\\
0 &0
\end{bmatrix},
\end{align*}
where $\tilde{\mfz}_{t} = \begin{bmatrix} \mfz_t^T & 1 \end{bmatrix}^T$. In particular, note that
\begin{align}
&\mE_{\mfx_{\nu - N + 2}|\mfN=N}\left[\tilde{\mfx}_{\nu - N + 2}\tilde{\mfx}_{\nu - N + 2}^T\right] \nonumber \\
& = \mE_{\mfx_{\nu - N + 1}|\mfN=N}[\tilde{\mfz}_{\nu - N + 1} \tilde{\mfz}_{\nu - N + 1}^T]  +
\begin{bmatrix}
\Sigma_w &0\\
0 &0
\end{bmatrix} \nonumber \\
& = 
\tilde{A}_{cl}(\nu - N + 1;\Qtrue,\qtrue)  \underbrace{\mE_{\bar{\mfx}|\mfN=N}[\bar{\tilde{\mfx}} \bar{\tilde{\mfx}}^T]}_{\succ 0}  \tilde{A}_{cl}(\nu - N + 1;\Qtrue,\qtrue)^T  \nonumber \\
& + 
\underbrace{
\begin{bmatrix}
\Sigma_w &0\\
0 &0
\end{bmatrix}}_{ \succeq 0} \succ 0, \label{eq:induction_start}
\end{align}
since $\tilde{A}_{cl}(t;\Qtrue,\qtrue)$ is invertible for all $t=1:\nu-1$ and since positive definiteness is invariant under congruence.
By induction, we thus have $\mE_{\mfx_t|\mfN=N}[\tilde{\mfx}_t\tilde{\mfx}_t^T]\succ 0$ for all $N$ such that $\mP(\mfN = N) > 0$, and in particular thus for $\mfN = \nu$ by Assumption~\ref{ass:planning_horizon}.

Now we show $\mE[\| \tilde{\mfx}_t \|^2]<\infty$. First, note that $\mE[\|\bar{\tilde{\mfx}}\|^2] = \mE[\|\bar{\mfx}\|^2]+1$.
By Assumption \ref{ass:persistent_excitation} we have $\mE[\|\bar{\mfx}\|^2]<\infty$, and hence $\mE[\|\bar{\tilde{\mfx}}\|^2]<\infty$.
In addition, by taking trace on both sides of \eqref{eq:induction_start}, moving the trace inside the expectation, rearranging terms, and using Cauchy-Schwartz inequality, we have that
\begin{align*}
& \mE_{\mfx_{\nu - N + 2}|\mfN=N}\left[ \| \tilde{\mfx}_{\nu - N + 2}\|^2 \right] \\
&  \leq 
\mE_{\bar{\mfx}|\mfN=N}[\| \bar{\tilde{\mfx}} \|^2 ] \cdot \| \tilde{A}_{cl}(\nu - N + 1;\Qtrue,\qtrue) \|_F^2 + \tr(\Sigma_w).
\end{align*}
Using an induction argument similar to the one above, we thus have that $\mE_{\mfx_{t}|\mfN=N} [ \| \tilde{\mfx}_{t} \|^2 ] < \infty$ for all $t = \nu - N + 1:\nu$ and all $N$ such that $\mP(\mfN = N) > 0$. Finally, 
\begin{align*}
& \mE[\| \tilde{\mfx}_t \|^2] =  \!\! \sum_{N = 1}^{\nu} \! \mP(\mfN=N) \mE_{\tilde{\mfx}_t| \mfN=N}[\|\tilde{\mfx}_t\|^2] < \infty,
\end{align*}
which proves the lemma.
\end{proof}

\begin{proof}[Proof of Theorem~\ref{thm:IOC_Q_optimal}]
As can be seen from the construction of the objective function in Section~\ref{subsec:construction_and_approximation}, and in view of  \eqref{eq:stochastic_control_stand_still},
\begin{align*}
& \Psi(Q, q,  P_{1:\nu}, \eta_{1:\nu},  \xi_{1:\nu-1}) + \sum_{t = 1}^{\nu-1}  \mE_{\mfx_t, \mfN}\left[\frac{1}{2} \|\mfu_t\|^2\right] \\
& = \! \sum_{N = 2}^{\nu}\mP(\mfN=N)\mE_{\mfy_{\nu-N+1:\nu}|\mfN=N}\left[\tilde{\psi}_N(\cdot) \right] \\
& \quad  + \sum_{N = 2}^{\nu}\mP(\mfN=N) \sum_{t = \nu - N + 1}^{\nu-1}  \mE_{\mfx_t \mid \mfN=N}\left[\frac{1}{2} \|\mfu_t\|^2\right] \\
& = \! \sum_{N = 2}^{\nu} \! \mP(\mfN=N) \!\!\!\!\!\!\! \sum_{t = \nu-N+1}^{\nu-1} \!\!\!\!\! \Big( \mE_{\mfx_{t}|\mfN=N}[\psi_{t,N}(Q,q;\mfx_t,\mfu_t) \! + \!  \frac{1}{2}\|\mfu_t\|^2 ] \Big)
\end{align*}
Now, by the definition of $\psi_{t,N}(Q,q;\mfx_t,\mfu_t)$ in \eqref{eq:psi_t_N}, as in \eqref{eq:obj_construct_x_t} we have that
\begin{align}
&\mE_{\mfx_t|\mfN=N}[\psi_{t,N}(Q,q;\mfx_t,\mfu_t)]\nonumber\\
&= \mE_{\mfx_t|\mfN=N}\Big[\mE_{\mfw_t}\Big[\frac{1}{2}(A\mfx_t+B\mfu_t+d+\mfw_t)^TP_{t+1}\nonumber\\
&\quad \times (A\mfx_t+B\mfu_t+d+\mfw_t)+\eta_{t+1}^T(A\mfx_t+B\mfu_t+d+\mfw_t) \Big]\nonumber\\
&\quad+\frac{1}{2}\mfx_t^TQ\mfx_t+q^T\mfx_t+\frac{1}{2}\|\mfu_t\|^2-\frac{1}{2}\mfx_t^TP_t\mfx_t-\eta_t^T\mfx_t\nonumber\\
&\quad+\frac{1}{2}g_t^T\mfR_t^\dagger g_t-\frac{1}{2}d^TP_{t+1}d-\eta_t^Td-\frac{1}{2}\tr(P_{t+1}\Sigma_w)\Big]. \nonumber
\end{align}
By opening the parenthesis and computing the expectation with respect to $\mfw_t$, and using Assumption~\ref{ass:IID}, we have that
\begin{align}
& \mE_{\mfx_t|\mfN=N}[\psi_{t,N}(Q,q;\mfx_t,\mfu_t)]\nonumber\\
&=\mE_{\mfx_t|\mfN=N}\Big[\frac{1}{2}(A\mfx_t+B\mfu_t+d)^TP_{t+1}(A\mfx_t+B\mfu_t+d)\nonumber\\
&\quad+\eta_{t+1}^T(A\mfx_t+B\mfu_t+d)-\frac{1}{2}\mfx_t^TP_t\mfx_t-\eta_t^T\mfx_t+\frac{1}{2}\xi_t\nonumber\\
&\quad+\frac{1}{2}\mfx_t^TQ\mfx_t+q^T\mfx_t\Big]\nonumber\\
&=\mE_{\mfx_t|\mfN=N}\Big[\frac{1}{2}
\begin{bmatrix}
\mfu_t^T &\mfx_t^T &1
\end{bmatrix}H_t
\begin{bmatrix}
\mfu_t \\ \mfx_t \\1
\end{bmatrix}
-\frac{1}{2}\|\mfu_t\|^2\Big],\label{eq:psi_t_N_rewrite}
\end{align}
where $H_t$ has the form \eqref{eq:nes_suff_existence_solution_2}. On the other hand, since $(Q,q,P_{1:\nu},\eta_{1:\nu},\xi_{1:\nu-1})$ is feasible, by the constraint \eqref{eq:stochastic_IOC_opt_pro_last_const} we have that $H_t\succeq 0$. Therefore, it holds that
\begin{align}
&\Psi(\cdot) =\sum_{N=2}^{\nu}\mP(\mfN=N)\sum_{t=\nu-N+1}^{\nu-1}
\mE_{\mfx_t|\mfN=N}\Big[\nonumber\\
&\qquad \frac{1}{2}
\begin{bmatrix}
\mfu_t^T &\mfx_t^T &1
\end{bmatrix}H_t
\begin{bmatrix}
\mfu_t \\ \mfx_t \\1
\end{bmatrix} -\frac{1}{2}\|\mfu_t\|^2\Big]\label{eq:Psi_rewrite}\\
& \geq - \sum_{N = 2}^{\nu}\mP(\mfN=N) \sum_{t = \nu - N + 1}^{\nu-1}  \mE_{\mfx_t \mid \mfN=N}\left[\frac{1}{2} \|\mfu_t\|^2\right] \nonumber \\
& = -  \sum_{t = 1}^{\nu-1}  \mE_{\mfx_t, \mfN}\left[\frac{1}{2} \|\mfu_t\|^2\right]. \nonumber
\end{align}
This proves the first part of the theorem.

Next, we show that the lower bound is actually attained by $(\bar{Q}, \bar{q},\Ptrue_{1:\nu},\etatrue_{1:\nu},\xitrue_{1:\nu-1})$. By using Theorem~\ref{thm:indefinite_LQR} we have that the true underlying $\bar{Q}$ and $\bar{q}$, together with corresponding solution $\{\bar{P}_{t} \in \mathbb{S}^n \}_{t = 1:\nu}$ and $\{ \bar{\eta}_{t} \in\mathbb{R}^{n} \}_{t = 1:\nu}$ to the Riccati recursions \eqref{eq:generalized_riccati_iterations}, and with $\xitrue_t = \gtrue^T_t \mfRtrue_t^\dagger \gtrue_t$ for $t=1:\nu-1$, is a feasible solution to the optimization problem, if $\varphi$ is large enough. For this feasible solution $(\bar{Q}, \bar{q},\Ptrue_{1:\nu},\etatrue_{1:\nu},\xitrue_{1:\nu-1})$, we can decompose the corresponding $\bar{H}_t$ as in \eqref{eq:H_decomp} and in this case, together with \eqref{eq:optimal_ctrl_formula_no_kernel}, \eqref{eq:psi_t_N_rewrite} can be written as
\begin{align*}
&\mE_{\mfx_t|\mfN=N}[\psi_{t,N}(Q,q;\mfx_t,\mfu_t)]\\
&=\mE_{\mfx_t|\mfN=N}\Big[\frac{1}{2}
\begin{bmatrix}
\mfu_t^T &\mfx_t^T &1
\end{bmatrix}\bar{H}_t
\begin{bmatrix}
\mfu_t \\ \mfx_t \\1
\end{bmatrix}
-\frac{1}{2}\|\mfu_t\|^2\Big]\\
\\
&=\mE_{\mfx_t|\mfN=N}\Big[\frac{1}{2}
\begin{bmatrix}
\mfu_t^T &\mfx_t^T &1
\end{bmatrix}
\begin{bmatrix}
\mfRtrue_t \\ \mfStrue_t^T \\ \gtrue_t^T
\end{bmatrix}
\mfRtrue_t^\dagger
\begin{bmatrix}
\mfRtrue_t & \mfStrue_t^T & \gtrue_t^T
\end{bmatrix}
\begin{bmatrix}
\mfu_t \\ \mfx_t^T \\1
\end{bmatrix}\\
&\quad -\frac{1}{2}\|\mfu_t\|^2\Big]\\
&=\mE_{\mfx_t|\mfN=N}\Big[\underbrace{\frac{1}{2}\|(\mfRtrue_t)^{\frac{1}{2}}(\mfRtrue_t\mfu_t+\mfStrue_t\mfx_t+\gtrue_t)\|^2}_{=0}-\frac{1}{2}\|\mfu_t\|^2\Big].
\end{align*}
This shows that the lower bound for the objective function $\Psi(\cdot)$ is attained by $(\bar{Q}, \bar{q},\Ptrue_{1:\nu},\etatrue_{1:\nu},\xitrue_{1:\nu-1})$.

Finally, we show that the ``true" $(\Qtrue,\qtrue,\Ptrue_{1:\nu}, \etatrue_{1:\nu}, \xitrue_{1:\nu-1})$ is actually the unique global optimizer to \eqref{eq:stochastic_IOC_opt_pro}.
To this end, let $(Q^\star,q^\star, P^\star_{1:N},\eta^\star_{1:N},\xi_{1:N-1}^\star)$ be an optimal solution to \eqref{eq:stochastic_IOC_opt_pro}, and let us also use $^\star$ to denote other vectors and matrices obtained using this optimal solution. Since the solution is optimal, it must be feasible, which implies that $H_t^\star\succeq 0,\forall t=1:\nu-1$. Hence it follows that $\mathfrak{R}_t^\star\succeq 0$, $\ker(\mathfrak{R}_t^\star)\subset \left[\ker(\mathfrak{S}_t^{\star T})\cap \ker(g_t^{\star T})\right]$ and $H_t^\star\backslash\mathfrak{R}_t^\star\succeq 0$, see \citep[Thm.~1.20, p.~43]{horn2005basic}. In view of the above ``kernel containment" and \eqref{eq:Psi_rewrite}, the optimal objective value can be further rewritten as
\begin{align*}
&\Psi(Q^\star,q^\star, P^\star_{1:\nu},\eta^\star_{1:\nu},\xi_{1:\nu-1}^\star) = \sum_{N=2}^{\nu}\mP(\mfN=N)\sum_{t=\nu-N+1}^{\nu-1}\\
& 
\mE_{\mfx_t|\mfN=N}\Big[ \frac{1}{2}
\begin{bmatrix}
\mfu_t^{T} &\mfx_t^{T} & 1
\end{bmatrix}
\begin{bmatrix}
I \\
\mathfrak{S}_t^{\star T}\mathfrak{R}_t^{\star \dagger} &I\\
g_t^{\star T}\mathfrak{R}_t^{\star \dagger}& &I
\end{bmatrix}\\
&\begin{bmatrix}
\mathfrak{R}_t^{\star} \\
&A^TP_{t+1}^\star A\!+\!Q^\star\!-\!P_t^\star \!-\!\mathfrak{S}_t^{\star T}\mathfrak{R}_t^{\star \dagger}\mathfrak{S}_t^\star &\beta_t^\star \!-\!\mathfrak{S}_t^{\star T}\mathfrak{R}_t^{\star\dagger}g_t^\star\\
&\beta_t^{\star T}\!-\!g_t^{\star T}\mathfrak{R}_t^{\star\dagger}\mathfrak{S}_t^\star &\xi_t^\star\!-\!g_t^\star\mathfrak{R}_t^{\star \dagger}g_t^\star
\end{bmatrix}\\
&\begin{bmatrix}
I &\mathfrak{R}_t^{\star\dagger}\mathfrak{S}_t^{\star \dagger} &\mathfrak{R}_t^{\star\dagger}g_t^{\star\dagger}\\
&I\\
& &I
\end{bmatrix}
\begin{bmatrix}
\mfu_t \\ \mfx_t \\ 1
\end{bmatrix}
- \frac{1}{2}\|\mfu_{t}\|^2\Big]
\end{align*}
Recalling the notation $\tilde{\mfx}_{t}=[\mfx_t^T,1]^T$ and the fact that $\mfR_t^\star\succeq 0$, we in turn get
\begin{align*}
&\Psi(\cdot)=\sum_{N=2}^\nu \mP(\mfN=N)\sum_{t=\nu-N+1}^{\nu-1}\mE_{\mfx_t|\mfN=N}\Big[ - \frac{1}{2}\|\mfu_{t}\|^2\\
&+\frac{1}{2}
\begin{bmatrix}
\mfu_t^{T}+\mfx_t^{T}\mathfrak{S}_t^{\star T}\mathfrak{R}_t^{\star \dagger}+g_t^{\star T}\mathfrak{R}_t^{\star\dagger} &\mfx_t^{T} & 1
\end{bmatrix}\\
&\times\begin{bmatrix}
\mathfrak{R}_t^{\star} \\
&H_t^\star\backslash\mathfrak{R}_t^\star
\end{bmatrix}
\begin{bmatrix}
\mfu_t+\mathfrak{R}_t^{\star\dagger}\mathfrak{S}_t^\star \mfx_t+\mathfrak{R}_t^{\star\dagger}g_t^\star\\
\mfx_t\\1
\end{bmatrix} \Big]\\
&=\sum_{t=1}^{\nu-1}\!\mE_{\mfx_t,\mfN}\!\Big[- \frac{1}{2}\|\mfu_{t}\|^2\Big]\!+\!\sum_{N=2}^\nu \!\mP(\mfN=N)\!\!\sum_{t=\nu-N+1}^{\nu-1}\mE_{\mfx_t|\mfN=N}\Big[\\
&\!\frac{1}{2} \|(\mfR_t^\star)^{\frac{1}{2}}(\mfu_t \! + \! \mfR_t^{\star\dagger}
\begin{bmatrix}
\mfS_t^\star &g_t^\star
\end{bmatrix} \! \tilde{\mfx}_t)\|^2\!+\!\frac{1}{2}\!\tr\left((H_t^\star\backslash\mfR_t^\star)\tilde{\mfx}_t\tilde{\mfx}_t^T\right)\Big],
\end{align*}
where $(\cdot)^{\frac{1}{2}}$ denotes the (uniquely defined) positive semi-definite matrix square root (see \citep[Thm.~7.2.6]{horn2013matrix}).
Note that since $H_t^\star\backslash\mathfrak{R}_t^\star\succeq 0$, all terms except $\sum_{t=1}^{\nu-1}\!\mE_{\mfx_t,\mfN}\!\left[- \frac{1}{2}\|\mfu_{t}\|^2\right]$ are non-negative.
Hence, in order for the lower bound $\sum_{t=1}^{\nu-1}\!\mE_{\mfx_t,\mfN}\!\left[- \frac{1}{2}\|\mfu_{t}\|^2\right]$ to be attained, we must have that 
\begin{subequations}
\begin{align}
&\mE_{\mfx_t|\mfN=N}\left[ \| (\mfR_t^\star)^{\frac{1}{2}}(\mfu_t+\mfR_t^{\star\dagger}
\begin{bmatrix}
\mfS_t^\star &g_t^\star
\end{bmatrix}\tilde{\mfx}_t) \|^2 \right]=0, \label{eq:u_opt_sol_form}\\
&\mE_{\mfx_t|\mfN=N}\left[\tr\left((H_t^\star\backslash \mfR_t^\star)\tilde{\mfx}_t\tilde{\mfx}_t^T\right)\right]\!=\!\tr\left(H_t^\star\backslash \mfR_t^\star\right)\!\!\!\mE_{\mfx_t|\mfN=N}\!\!\left[\tilde{\mfx}_t\tilde{\mfx}_t^T\right]\nonumber\\
&\qquad=0, \quad t=\nu-N+1:\nu-1,\label{eq:riccati_opt_sol_form}
\end{align}
\end{subequations}
for all $N$ such that $\mP(\mfN=N)>0$. In particular, by Assumption~\ref{ass:planning_horizon} it must be true for $\mfN = \nu$.
From Lemma~\ref{lem:stochastic_persistent_excitation}, we know that $\mE_{\mfx_t|\mfN=\nu}[\tilde{\mfx}_t\tilde{\mfx}_t^T]\succ 0$.
Thus it follows that, for $N=\nu$, \eqref{eq:riccati_opt_sol_form} implies that $H_t^\star\backslash \mathfrak{R}_t^\star=0$ holds for $t=1:\nu-1$.  By using the observation in Remark~\ref{rem:rank_psd_and_riccati}, we therefore have that $(Q^\star,q^\star, P^\star_{1:\nu},\eta^\star_{1:\nu})$ satisfies the generalized Riccati iterations \eqref{eq:generalized_riccati_iterations}.

Now, to show that the optimal solution to \eqref{eq:stochastic_IOC_opt_pro} is unique, first assume that $(Q^\star,q^\star, P^\star_{1:\nu},\eta^\star_{1:\nu})$ is an optimal solution such that $\mfR_t^\star \succ 0$, for $t=1:\nu-1$. In this case, also $(\mfR_t^\star)^{\frac{1}{2}}\succ 0$ for $t=1:\nu-1$.
By \eqref{eq:u_opt_sol_form}, this means that, conditioned on $\mfN = \nu$, we have $\mfu_t = -\mfR_t^{\star\dagger}
\begin{bmatrix}
\mfS_t^\star &g_t^\star
\end{bmatrix}\tilde{\mfx}_t$ a.s.~for $t = 1:\nu-1$. But conditioned on $\mfN = \nu$, we also have that $\mfu_t = -\mfRtrue_t^{\dagger}
\begin{bmatrix}
\mfStrue_t &\gtrue_t
\end{bmatrix}\tilde{\mfx}_t$. Therefore
\[
\mfR_t^{\star\dagger}
\begin{bmatrix}
\mfS_t^\star &g_t^\star
\end{bmatrix}\tilde{\mfx}_t = \mfRtrue_t^{\dagger}
\begin{bmatrix}
\mfStrue_t &\gtrue_t
\end{bmatrix}\tilde{\mfx}_t, \text{ a.s.,}
\]
for $t=1:\nu-1$. Multiplying from the right with $\tilde{\mfx}_t^T$ and taking expectation $\mE_{\mfx_t|\mfN=\nu}$ on both sides, we have that
\[
\mfR_t^{\star\dagger}
\begin{bmatrix}
\mfS_t^\star &g_t^\star
\end{bmatrix} \mE_{\mfx_t|\mfN=\nu}[\tilde{\mfx}_t \tilde{\mfx}_t^T] = \mfRtrue_t^{\dagger}
\begin{bmatrix}
\mfStrue_t &\gtrue_t
\end{bmatrix}\mE_{\mfx_t|\mfN=\nu}[\tilde{\mfx}_t \tilde{\mfx}_t^T],
\]
for $t=1:\nu-1$. Using Lemma~\ref{lem:stochastic_persistent_excitation}, we know that $\mE_{\mfx_t|\mfN=\nu}[\tilde{\mfx}_t \tilde{\mfx}_t^T] \succ 0$ and hence the matrix is full rank. Therefore, it must hold that
\[
\mfR_t^{\star\dagger}
\begin{bmatrix}
\mfS_t^\star &g_t^\star
\end{bmatrix} = \mfRtrue_t^{\dagger}
\begin{bmatrix}
\mfStrue_t &\gtrue_t
\end{bmatrix}, \quad t=1:\nu-1,
\]
and thus, by \eqref{eq:tildeAcl}, that
\[
\tilde{A}_{cl}(t;Q^\star,q^\star)  = \tilde{A}_{cl}(t;\Qtrue,\qtrue), \quad t=1:\nu-1.
\]
By Proposition~\ref{prop:structural_identifiability}, we therefore have that $\Qtrue=Q^\star$ and that $\qtrue = q^\star$.  This implies that in the subset of the feasible region 
\eqref{eq:stochastic_IOC_opt_pro_first_const}-\eqref{eq:stochastic_IOC_opt_pro_last_const} where $\mfR_t \succ 0$, for $t=1:\nu-1$, it holds that $(\Qtrue, \qtrue, \Ptrue_{1:\nu}, \etatrue_{1:\nu}, \xitrue_{1:\nu-1})$ is the unique globally optimal solution to \eqref{eq:stochastic_IOC_opt_pro}.
Next, suppose that there exists a minimizer $(Q^{\star},q^{\star},P_{1:\nu}^{\star},\eta_{1:\nu}^{\star},\xi_{1:\nu-1}^{\star})$ such that $\mfR_t^\star \succeq  0$ but not strictly positive definite for some $t\in\{1,\ldots,\nu-1\}$. In particular, this means that $(Q^{\star},q^{\star},P_{1:\nu}^{\star},\eta_{1:\nu}^{\star},\xi_{1:\nu-1}^{\star}) \neq (\Qtrue, \qtrue,\Ptrue_{1:\nu},\etatrue_{1:\nu},\xitrue_{1:\nu-1})$.
Moreover, since \eqref{eq:stochastic_IOC_opt_pro} is a convex optimization problem that attains an optimal solution, the set of all optimal solutions is a nonempty convex set \citep[Thm.~27.2]{rockafellar1970convex}.
This means that $(\alpha\bar{Q}+(1-\alpha)Q^\star, \alpha\bar{q}+(1-\alpha)q^\star,\{\alpha\Ptrue_t+(1-\alpha)P_t^\star\}_{t=1}^\nu,\{\alpha\etatrue_t+(1-\alpha)\eta_t^\star\}_{t=1}^\nu,\{\alpha\xitrue_t+(1-\alpha)\xi_t^\star\}_{t=1}^{\nu-1})$ are all optimal for all $\alpha\in[0,1]$; for each $\alpha \in [0,1]$ we denote the optimal solution $(Q^{\alpha},q^{\alpha},P_{1:\nu}^{\alpha},\eta_{1:\nu}^{\alpha},\xi_{1:\nu-1}^{\alpha})$.
Since the eigenvalues of $\mfR_t$, $t=1:\nu-1$, depends smoothly on $P_t$ (see \eqref{eq:psd_kernel_containment_cond}), we can select $\alpha$ close enough to 1 so that $(Q^{\alpha},q^{\alpha},P_{1:\nu}^{\alpha},\eta_{1:\nu}^{\alpha},\xi_{1:\nu-1}^{\alpha})$ will be such that $\mfR_t^{\alpha} \succ 0$ for all $t=1:\nu-1$. However, this contradicts the fact that $(\Qtrue, \qtrue, \Ptrue_{1:\nu}, \etatrue_{1:\nu}, \xitrue_{1:\nu-1})$ is the unique globally optimal solution to \eqref{eq:stochastic_IOC_opt_pro} with $\mfR_t\succ 0$, $t=1:\nu-1$.
Therefore, there can be no optimal solution such that  $\mfR_t^\star$ is not (strictly) positive definite for all $t\in\{1,\ldots,\nu-1\}$, and hence $(\Qtrue, \qtrue, \Ptrue_{1:\nu}, \etatrue_{1:\nu}, \xitrue_{1:\nu-1})$ is the unique globally optimal solution to \eqref{eq:stochastic_IOC_opt_pro}.
\end{proof}

\bibliographystyle{plain}
\bibliography{}

\end{document}